\documentclass[11pt]{article}
\usepackage[margin=1.3in]{geometry}
\usepackage{graphicx} % Required for inserting images
\usepackage{authblk}  % For better author formatting
\usepackage{hyperref} % For clickable email links
\usepackage[dvipsnames]{xcolor}
\usepackage{esvect}
\usepackage[ruled,vlined,linesnumbered]{algorithm2e}
\SetAlgoNoLine
\usepackage{graphicx}      
\usepackage{amsmath,amssymb, amsfonts,amsthm}

\usepackage{subcaption}
\usepackage{standalone}
\usepackage{hyperref}   
\usepackage{color}% for hyperlinks
\newtheorem{Theorem}{Theorem}
\newtheorem{Lemma}{Lemma}
\newtheorem{Remark}{Remark}
\newtheorem{proposition}[Theorem]{Proposition} 
\usepackage{tikz}
\usetikzlibrary{arrows.meta,calc}
\usepackage{float}
\usepackage{todonotes}
\usepackage[dvipsnames]{xcolor}
\usepackage{enumitem}
\usepackage{url}
\usepackage{array}
\usepackage{caption}
\usepackage{nicematrix}
\usepackage{arydshln}  
\usepackage{array} % Required for tabular
\usepackage{tikz}
\usepackage{pgfplots}
\pgfplotsset{compat=1.8}
\usepackage{pgfplotstable}
\usetikzlibrary{3d, calc}
\usepackage{placeins}   % Ensures figures do not float past barriers
\usepackage{cases}
\pgfplotsset{compat=1.17}

\usepackage[
    backend = biber,
    doi = false,
    url = false,
    isbn = false,
    sorting = none,
    giveninits = true,  % modern replacement for firstinits
    maxnames = 99,
    hyperref = true,
    block = none
]{biblatex}

\renewbibmacro{in:}{}

\bibliography{reference}

\DeclareMathOperator{\diag}{diag}

\newcommand{\R}{\mathbb{R}}

\newcommand{\rint}{\operatorname{int}}

\newcommand{\xc}[1]{\vspace{.1cm}}

\newcommand{\dist}{\operatorname{dist}}

\DeclareMathOperator{\sgn}{sgn}
\DeclareMathOperator{\nll}{null}

\makeatletter
\def\blfootnote{\xdef\@thefnmark{}\@footnotetext}
\makeatother

\title{On Coexistence of Conservative Replicator Dynamics with Four Strategies}

\begin{document}
\author{ Haoyu Yin$^*$, \quad Xudong Chen, \quad and \quad Bruno Sinopoli 
}
\date{}
\maketitle
\blfootnote{H.Yin and X.Chen are with Department of Electrical and Systems Engineering, Washington University in St. Louis. B.Sinopoli is with School of Electrical, Computer and Energy Engineering, 
Arizona State University. Emails: \texttt{\{h.yin,  cxudong\}@wustl.edu}, \texttt{\{bruno.sinopoli\}@asu.edu}. }

\blfootnote{$^*$Corresponding author: H.Yin}
%\blfootnote{This work was supported by NSF ECCS-2426017 and by DARPA HR0011-25-3-0211.}

\begin{abstract}                          % Abstract of not more than 200 words.
In this paper, we study four-strategy conservative replicator dynamics induced by constant payoff matrices. 
We establish necessary and sufficient conditions for 
coexistence to occur by associating the payoff matrix with its digraph, revealing exactly five distinct digraph classes governing the global behavior. We further show that, whenever the dynamics is coexistent, every non-equilibrium trajectory in the relative interior of the simplex is a Lyapunov-stable periodic orbit.
Together with the classification of the boundary phase portraits, these results provide a complete characterization of the global dynamics in the four-strategy case with coexistence.
\end{abstract}

% \end{frontmatter}

\section{Introduction}\label{sec:intro}
Over the past decades, replicator dynamics has emerged as a fundamental model in evolutionary game theory for describing the continuous evolution of strategy shares in a large population driven by payoff differences~\cite{AkinLosert1984,HofbauerSigmund1998}. Let \(x=(x_1,\dots,x_n)\) denote the population state, evolving on the standard simplex: 
$$\Delta^{n-1}:=\left \{x\in\mathbb{R}^n_{\ge 0} \mid \sum_{i=1}^n x_i=1 \right \}.$$
Given a payoff matrix \(A\in\mathbb{R}^{n\times n}\) that encodes the fitness/payoff of each strategy against the current population state, the replicator dynamics is given by
\begin{equation}\label{eq:intro_rep}
\dot x_i(t) = x_i(t)\left ((Ax(t))_i - x^\top(t) A x(t) \right ),
\end{equation}
for all $i=1,\dots,n$. 
In the dynamics, the growth rate of each strategy is determined by its excess payoff relative to the population average~\cite{Sandholm2010}. In particular, strategies performing better than average increase instantly in proportion, whereas those performing worse decrease. This model has become a standard framework for studying strategic adaptation and long-run selection~\cite{HofbauerSigmund2003}, as it links the geometry of the simplex with the incentive structure induced by the game specified by payoff matrix $A$.

% {\color{red} add more references; review papers}
{ 
Replicator dynamics has been widely appreciated in population games for its use as an aggregate model for boundedly rational learning among interacting agents and has been extensively investigated in the literature (see, e.g.,~\cite{HofbauerSigmund2003,mertikopoulos2016learning,como2020imitation,hankins2026population,yin2026zero,quijano2017role} and references therein). %This viewpoint has been used across population games and related application domains. 
In applications such as resource-allocation problems, $x_i(t)$ can represent the fraction of agents selecting a certain task, route, or resource at time~$t$, while payoffs encode congestion, performance, or utility feedback~\cite{quijano2017role,barreiro2016distributed,ramirez2010population}. In epidemic game models, strategy shares can represent the adoption of protective behaviors such as vaccination, quarantine, or social distancing, with payoffs determined by the trade-off between infection risk and behavioral cost~\cite{amaral2021epidemiological,martins2023epidemic}. More broadly, in networked multi-population settings, each population (or community) may carry its own strategy distribution, and these distributions evolve through payoff-based revision rules while being coupled through inter-population interactions and subject to constraints on shared resources, mobility, and network topology~\cite{riehl2018survey,como2020imitation,govaert2022population}.

We also mention a few models that extend the standard replicator dynamics. 
The so-called co-evolving evolutionary games~\cite{gong2022limit,skoulakis2021evolutionary} study the interplay between the population's strategic behavior and the surrounding environment by introducing a new variable, namely, the environment state, and investigate the dynamic behavior of the coupled system.   
%allow the population state and an environmental state to evolve jointly, so that the payoff landscape changes in response to the current population distribution. 
%Networked evolutionary dynamics incorporate graph-mediated interactions among agents, populations, or communities, making collective behavior depend on both payoff incentives and interconnection topology~\cite{riehl2018survey}. 
Next, we mention higher-order learning dynamics, which enrich payoff-based revision by allowing strategy updates to depend on auxiliary learning states, such as accumulated payoff information~\cite{gao2020passivity,laraki2013higher}, memory~\cite{fox2013population}, or filtering variables~\cite{mabrok2016passivity,park2019payoff}, rather than only on the instantaneous payoff vector. We further mention the replicator-mutator dynamics that incorporates mutation, exploration, or noisy strategy revision. This model has been used for scenarios where agents not only imitate successful strategies but may also switch to alternative strategies through error, experimentation, or imperfect transmission, such as language evolution, behavior adoption in social networks, and decision-making in networked multi-agent systems~\cite{pais2012hopf}.
}

{  
A complementary line of work treats the evolutionary process as a system to be controlled or planned. Rather than focusing on characterizing the behavior of the dynamics, these studies introduce interventions that modify incentives, payoff feedback, or interaction structure in order to steer the population toward a desired state. In networked evolutionary games, control inputs can be applied to selected agents or interactions to influence the collective state of the population~\cite{riehl2016towards,riehl2018survey}. Adaptive-gain mechanisms instead adjust the strength of payoff feedback along the evolution, thereby shaping the rate and direction of adaptation under replicator dynamics~\cite{zino2023adaptive}. More recently, social-planning formulations have treated relative fitness as a design variable, allowing a planner to influence how boundedly rational agents learn and to align the resulting population behavior with a system-level objective~\cite{siththaranjan2025social}. 
}
% { 
% The replicator dynamics~\eqref{eq:intro_rep} also admits a game-theoretic interpretation through the lens of population games~\cite{Sandholm2010}. In this interpretation, $x_i$ represents the fraction of a large population currently using strategy $i$, and $A$ specifies the payoff landscape faced by boundedly rational agents. Rather than solving a global optimization problem, agents revise their strategies through simple payoff-based learning or imitation rules.
% The replicator dynamics is one canonical aggregate model generated by such revision protocols: strategies with above-average payoff gain population share, while those with below-average payoff lose share. 
% This viewpoint connects the geometric analysis of~\eqref{eq:intro_rep} to a broader class of learning dynamics in population games.
% }

\subsection{Scope of our work and main contributions}
In this paper, we will focus on the class of {\it conservative} replicator dynamics.   
%extends the zero-sum games~\cite{AkinLosert1984,yin2026zero} to a broader family of non-dissipative systems}. 
Such dynamics admit conserved quantities and can exhibit recurrent behavior. A natural question in this context is whether all strategies persist in the long run. This is captured by the notion of {\it coexistence}, which requires that every trajectory of~\eqref{eq:intro_rep}, whose initial condition $x(0)$ is in $\rint \Delta^{n-1}$, i.e., the relative interior of $\Delta^{n-1}$, is asymptotically bounded away from the boundary $\partial \Delta^{n-1}$. More precisely, the dynamics~\eqref{eq:intro_rep} is coexistent if for every $x(0)\in\rint \Delta^{n-1}$, there exists an $\varepsilon>0$ such that 
$$
\lim_{t\to\infty} \inf x_i(t)\ \ge\ \varepsilon,\quad \mbox{for all } i=1,\dots,n.
$$

Understanding when coexistence can arise is a fundamental problem in evolutionary game theory and population dynamics, %as it formalizes robust coexistence: 
as it rules out extinction of initially present strategies as time evolves. 
%and ensures that the asymptotic dynamics does not collapse onto lower-dimensional faces of the simplex. 
%This question is further motivated by the broad range of conservative and related population-dynamic models 
Its importance has also been widely appreciated in many applications, including condensate selection in driven-dissipative bosonic systems~\cite{knebel2015evolutionary}, graph-structured evolutionary dynamics~\cite{madeo2014game}, and distributed population-dynamics methods for optimization and control synthesis~\cite{barreiro2016distributed,quijano2017role}. 
Necessary and/or sufficient conditions for coexistence have been obtained in the literature. For example, in~\cite{knebel2013coexistence}, the authors have shown that dynamics~\eqref{eq:intro_rep} with skew-symmetric $A$ is coexistent if and only if the null space of $A$ intersects $\rint \Delta^{n-1}$, i.e., the set of all strictly positive probability vectors. A literature review will be provided shortly.  

However, as seen from its definition, coexistence alone does not tell how each individual trajectory of the dynamics~\eqref{eq:intro_rep} evolves over time, unless the dynamics~\eqref{eq:intro_rep} is of a special type, such as gradient flow whose stable equilibria belong to the interior of $\Delta^{n-1}$, or,   
the state space of the dynamics is two-dimensional (specifically, $n = 3$ and the replicator dynamics evolves on $\Delta^2$). For the latter case, one can use the celebrated Poincar\'e-Bendixson theorem to conclude that if the system is coexistent, then each trajectory with $x(0)\in \rint \Delta^2$ converges to either an equilibrium point or to a periodic orbit. 
The problem quickly becomes intractable as $n$ increases. The literature is very sparse for characterizing phase portraits of coexistent replicator dynamics for $n \geq 4$. 

This paper deals with conservative (and hence, non-gradient) replicator dynamics~\eqref{eq:intro_rep} with four strategies (i.e., $n = 4$). We provide a necessary and sufficient condition for coexistence of such replicator dynamics and, moreover, a complete characterization of the phase portrait of any such system.  
As shown in~\cite{HofbauerSigmund1998}, every payoff matrix $A$ in this class is dynamically equivalent, up to column shifts, to a skew-symmetric matrix. %(we note here that every polynomial replicator dynamics can arise as a population game with skew-symmetric payoff, which.
It therefore suffices to characterize the dynamics induced by skew-symmetric matrices $A\in \R^{4\times 4}$ as
\begin{equation}
\label{eq:zsreplicator}
    \dot x_i(t) =x_i(t)(Ax(t))_i,\quad \mbox{for all } i=1,2,3,4,
\end{equation}
where $x^\top(t) Ax(t)=0$. 
Note that the above system can be translated to a three-dimensional antisymmetric Lotka-Volterra system~\cite{AkinLosert1984}. %, which arise naturally in four-species cyclic competition. 
The state evolution describes the change of population shares of four competing types, and can exhibit rich geometric behavior, including saddles, spirals, and straight-line trajectories as shown 
in~\cite{durney2011saddles}.
{ Such systems arise in closed biochemical reaction networks with conserved total concentration, where the four coordinates represent the concentrations of four molecular states connected by a cyclic sequence of one-step transitions, commonly described as a macromolecular ``turning wheel''; see~\cite{doi:10.1073/pnas.85.16.5923,murza2010global}. %Conservation of total concentration places the normalized state on $\Delta^3$, with periodic interior orbits representing persistent oscillations among the four states and boundary convergence indicating the depletion of one or more states.
} The dynamics also appears through the standard von Neumann symmetrization of a $2\times 2$ two-player zero-sum game, such as matching pennies, making it a canonical benchmark for cyclic, non-convergent learning dynamics~\cite{papadimitriou2016nash,biggar2024attractor}.

% In the applications discussed above, the four coordinates can represent four competing resources, behaviors, or strategic choices. The conservative four-strategy model isolates the dynamics generated by cyclic payoff interactions before environmental feedback, network coupling, mutation, learning memory, or control interventions are added. The characterization of its phase portrait therefore provides a baseline for determining whether cyclic incentives sustain coexistence and recurrent behavior or drive the population toward a lower-dimensional strategy set.}
%In these contexts, the four-dimensional system acts as a foundational network motif: it has the minimal state space capable of sustaining coupled transition loops and non-adjacent cross-communications~\cite{lutz2013intransitivity}. 

%\subsection{Main contributions}
As mentioned earlier, a necessary condition for coexistence is that $A$ has a nontrivial null space, i.e., $\det(A) = 0$. %Thus, in this paper, we focus on the case where $\det(A) = 0$. 
Since $A$ is a $4$-by-$4$ skew-symmetric matrix, either $A = 0$ or the rank of $A$ is $2$. The former case is trivial, and we will address the latter. 
The main contributions of this work are threefold: 
\begin{enumerate}
\item First, in Section~\ref{sec:charc_coexistence}, we establish a necessary and sufficient condition on $A$ for coexistence of~\eqref{eq:zsreplicator} by introducing an appropriate digraph $\mathcal{G}_A$ induced by the sign pattern of $A$. We show that~\eqref{eq:zsreplicator} is coexistent if and only if $\det A = 0$ and the graph $\mathcal{G}_A$ has a directed cycle. 
%We present this result and its proof in .
% the payoff matrix \(A\) with its directed graph \(\mathcal{G}_A\). This yields exactly five digraph isomorphism classes of the sign patterns for the entries in payoff matrix. 

\item Then, in Section~\ref{sec:sys_interior}, we establish the fact that if the system~\eqref{eq:zsreplicator} is coexistent, then any $x$ in $\rint \Delta^3$ is either an equilibrium point or belongs to a stable periodic orbit. 
We develop a new approach to prove the result. Specifically, we introduce a one-parameter family of invariant functions (i.e., functions that are forward invariant along the dynamics), whose iso-level sets are diffeomorphic to $2$-spheres. Then, we show that 
every trajectory $x(t)$ of the replicator dynamics~\eqref{eq:zsreplicator}, with $x(0)$ in $\rint \Delta^3$, belongs to the iso-level sets of two invariant functions (dependent on $x(0)$) which intersect each other transversally. The preimage theorem then implies that every connected component of the intersection is diffeomorphic to a circle $S^1$, so $x(t)$ belongs to a periodic orbit. The stability of such periodic orbits is established by constructing a suitable Lyapunov function.

\item Finally, in Section~\ref{sec:sys_boundary}, we characterize the phase portrait of dynamics~\eqref{eq:zsreplicator} on $\partial \Delta^3$. Note that $\Delta^3$ has four faces, each of which is the standard $2$-simplex. The dynamical behavior of replicator dynamics on $\Delta^2$ is very well known. 
In fact, seminal works by Zeeman~\cite{Zeeman1980} and Bomze~\cite{Bomze1983,Bomze1995}
have classified all possible phase portraits for~\eqref{eq:intro_rep} with general $3$-by-$3$ payoff matrices $A$. We do not intend to replicate existing efforts, but rather the contribution of Section~\ref{sec:sys_boundary} is to enumerate all possible combinations of these phase portraits for the four faces of $\Delta^3$. Amongst others, we relate the digraph $\mathcal{G}_A$ to the classification problem, showing that the allowable graph topologies and the possible combinations one-to-one correspond to each other.  
% In particular, we have classified all the phase portraits for conservative coexistent replicator dynamics with $n=4$.

% the boundary dynamics for each of the five digraph isomorphism classes, and show that systems with isomorphic induced digraphs have the same qualitative behavior on \(\Delta^3\).
\end{enumerate}

\subsection{Literature review}

References~\cite{Weibull1995,HofbauerSigmund1998,Sandholm2010,HofbauerSigmund2003} have established dynamical properties of replicator dynamics, including forward invariance of the simplex and its faces, equilibrium and stability notions, and Lyapunov-based convergence results for important subclasses such as potential games~\cite{Sandholm2010}. 

In the three-strategy case, the state space is planar, which makes it possible to combine the properties of replicator dynamics with classical two-dimensional dynamical-systems techniques to derive an explicit phase-portrait classification; see~\cite{Bomze1983,Bomze1995}. % The classification of phase portraits on the $2$-simplex was initiated and completed in~\cite{Zeeman1980}, revealing 33 generic phase portraits (19 up to flow reversal), while non-generic cases may contain families of periodic orbits.
{ In particular, the authors of~\cite{BoonePiliouras2019WINE} have shown that, for three-strategy zero-sum replicator dynamics admitting an interior Nash equilibrium, every trajectory starting from an interior point either remains at an equilibrium or is periodic. Their proof uses a conserved quantity to keep non-equilibrium interior trajectories away from the boundary and then applies the Poincar\'e-Bendixson theorem~\cite{schwartz1963generalization} to characterize the limit sets as periodic orbits.
}

% any recurrent replicator dynamics with three strategies,++++++ the recurrence is exactly a periodic cycle 
% The proof built upon Poincar\'e-Bendixson Theorem~\cite{schwartz1963generalization}. }
%However, these techniques do not extend directly to the case \(A\in\mathbb{R}^{4\times4}\), where the state space is three-dimensional.

For conservative and zero-sum replicator systems with more than three strategies, the literature has mainly focused on the existence of first integrals and the associated recurrent behavior of the dynamics~\cite{AkinLosert1984,mertikopoulos2018riemannian}, rather than a complete geometric description of trajectories.
In particular, conservative replicator dynamics is closely related to Hamiltonian and volume-preserving systems, and recurrence is often established through invariant quantities together with the Poincar\'e recurrence theorem
\cite{barreira2006poincare}. However, 
recurrence alone may not determine the orbits of the dynamical systems. It remains largely open whether there exist periodic orbits. This issue has also been raised in recent work in online learning and algorithmic game theory, where recurrent and cyclic behavior in zero-sum games plays a central role \cite{cheung2019vortices,MertikopoulosPapadimitriouPiliouras2018,DaskalakisPanageas2019}.

{ 
Perhaps the closest work to ours is~\cite{murza2010global}, where the authors have analyzed the global dynamics of a specific three-dimensional cyclic Lotka-Volterra family, which can be transformed into four-strategy replicator dynamics through a linear change of variables. 
In the notation of the present paper, their family corresponds to a special case of ours, namely, Case~(\ref{subfig:V}) in Fig.~\ref{fig:combined_cases}. 
Their result (Theorem 1.1) indicates that the interior of the simplex $\Delta^3$ is foliated by periodic orbits around an equilibrium segment.  
The proof goes by finding certain first integrals that confine the flows to two-dimensional invariant level surfaces, which then enables the use of the Poincar\'e-Bendixson theorem. 
% Given the specific payoff structure, the dynamics can be reduced to a planar stting where the Poincar\'e-Bendixson Theorem can be  applied. 
% In particular, 
Their analysis depends on the specific structure of the payoff matrix and cannot be extended to the other cases (i.e., Cases~(\ref{subfig:I})-(\ref{subfig:IV})) in our paper.  
%While~\cite{murza2010global} gives a global analysis of this particular subfamily, our results cover all four-strategy conservative replicator dynamics whose induced digraph contains a directed cycle and provide a unified phase-portrait classification on $\Delta^3$. 

For other relevant works, we mention~\cite{pais2012hopf,gong2022limit}, in which the authors have used the Hopf bifurcation to establish existence of periodic orbits and to identify parameter regimes where stable limit cycles can arise in the replicator-mutator dynamics. 
%from the mutation terms, network coupling, or other perturbations of the classical replicator dynamics~\cite{pais2012hopf}. 
Their models differ significantly from ours. By the nature of bifurcation theory, their analysis is local and parameter-dependent and thus, does not address the problem of this paper. %For evolutionary dynamics with dynamic payoff models, passivity has been used to view the payoff mechanism and the population dynamics as interconnected input-output systems. Under suitable passivity or dissipativity conditions, storage-function arguments yield stability and convergence of the closed-loop system to equilibrium~\cite{park2018passivity,park2019payoff}. In co-evolving evolutionary games, the population state is coupled to an environmental state where Lyapunov and bifurcation-based methods have been used to characterize stability, limit cycles, and more complicated asymptotic behavior generated by this feedback loop~\cite{gong2022limit}. 
%These tools, however, do not directly yield a global characterization for the dynamics we study in this paper.
}

Finally, we mention a line of research that connects the long-time behavior of zero-sum dynamics to graph topology.
Starting from the characterization of coexistence via a strictly positive kernel of the skew-symmetric payoff matrix~\cite{knebel2013coexistence},
the authors in~\cite{geiger2018topologically} have identified network topologies that enforce coexistence independently of the precise interaction strengths for the replicator dynamics with an odd number of strategies,  and related the condition to the Pfaffian orientations of the payoff matrix~\cite{geiger2018topologically}.
In parallel, graph-based operations have been used to generate integrable families of zero-sum replicator systems, which have enough first integrals to make the dynamics explicitly analyzable:~\cite{EvripidouKassotakisVanhaecke2022} developed a graph-theoretic classification in terms of irreducible weighted graphs, morphisms, and decloning operations,
while~\cite{PaikGriffin2023} has shown that suitable graph embeddings produce new completely integrable skew-symmetric replicator systems.
More recently, response graphs~\cite{BiggarShames2023SciRep} and preference graphs~\cite{biggar2025sink} have been used to characterize long-run behavior in game dynamics by relating graph-theoretic sink components~\cite{papadimitriou2016nash} to attractors and chain-recurrent sets. However, a complete characterization of coexistence and of the resulting global dynamics remains unavailable.

The remainder of this paper is organized as follows. We first gather below the notations required for the analysis. Section~\ref{sec:charc_coexistence} establishes necessary and sufficient conditions for the coexistence of four-strategy conservative replicator dynamics (Section~\ref{ssec:iff_coexistence}) and introduces five digraph isomorphism classes to reveal exactly how the null space of $A$ intersects the simplex (Section~\ref{ssec:finer_class}).  Section~\ref{sec:sys_interior} investigates behavior in $\rint \Delta^3$, rigorously proving that any non-equilibrium interior trajectory is a stable periodic orbit. Section~\ref{sec:sys_boundary} provides a complete classification of phase portraits on the boundary based on the isomorphism classes of the digraph. The paper ends with a summary and outlook in Section~\ref{sec:conclusion}.

% \subsection{Why $4\times4$ skew-symmetric games are the first nontrivial frontier}

\subsection{Notation}
Let $A\in\mathbb{R}^{4\times 4}$ be a skew-symmetric matrix, with entries denoted by $a_{ij}=-a_{ji}$ for $i,j\in\{1,2,3,4\}$, and $a_{ii}=0$. The determinant of $A$ is given by
\begin{equation*}
\det(A)=(a_{12}a_{34}-a_{13}a_{24}+a_{14}a_{23})^2.
\end{equation*}
Let $\mathbf{1}\in\mathbb{R}^4$ denote the vector of all ones. The standard simplex $\Delta^3$, whose vertices correspond to the standard basis vectors $e_i$, is embedded in the affine hyperplane $H:=\{x\in\mathbb{R}^4\mid\mathbf{1}^\top x=1\}$. We denote its relative interior by $\rint \Delta^3:=\{x\in\Delta^3\mid x_i>0\text{ for all }i\}$, and its boundary by $\partial\Delta^3:=\{x\in\Delta^3\mid x_i=0\text{ for some }i\}$. Both the simplex and its boundary are forward invariant under~\eqref{eq:zsreplicator}. Finally, we denote the set of interior equilibria by
$$K:=\nll(A)\cap\rint \Delta^3 .$$

% \begin{remark}\label{rem:faces}
% If $\mathcal{I}$ is unbounded toward $0$ and $+\infty$, then the endpoints correspond to the limits $t \to 0$ and $t \to \infty$, lying on the faces $\{x_4=0\}$ and $\{x_3=0\}$, respectively. Generally (when $\mathcal{I}$ is bounded), endpoints may occur at finite $t$ where $x_1=0$ or $x_2=0$. Lemma~\ref{lem:segment} holds regardless of which faces contain the endpoints.
% \end{remark}
\section{Characterization of Coexistence}
\label{sec:charc_coexistence}
It is known that zero-sum replicator dynamics is coexistent if and only if $K$ is non-empty. For the dynamics in~\eqref{eq:zsreplicator}, this condition translates into strong structural constraints. In this section, we reformulate these constraints in graph-theoretic terms, first characterizing coexistence and then classifying all digraph topologies that give rise to coexistence for the dynamics.

\subsection{Necessary and sufficient condition for coexistence}
\label{ssec:iff_coexistence}

% {  graph using mathcal, geometry using normal, change the $K$ to $K$ }
% \subsection{}
In this section, we provide a necessary and sufficient condition for the replicator dynamics~\eqref{eq:zsreplicator}, with $A$ skew-symmetric, to be coexistent. To this end, we associate to $A$ the simple digraph $\mathcal{G}_A = (\mathcal{V}, \mathcal{E})$ on $4$ nodes whose node set is $\mathcal{V}=\{1,2,3,4\}$ and whose edge set is 
$$
\mathcal{E}:=\{(i,j)\in \mathcal{V}\times \mathcal{V} \mid  a_{ij}<0\}.
$$  
Note, in particular, that for any two distinct nodes $i$ and $j$, there exists at most one edge incident to both of them. Thus, if $\mathcal{C}$ is a cycle of $\mathcal{G}_A$, then the length of $\mathcal{C}$ is either $3$ or $4$. 
We now state the main result of this section:

\begin{Theorem}\label{thm:main2}
Let $A\in \R^{4\times 4}$ be a nonzero skew-symmetric matrix and $\mathcal{G}_A$ be the associated digraph. 
Then, the replicator dynamics~\eqref{eq:zsreplicator} is coexistent if and only if the following two items hold: 
\begin{enumerate}
% %[label=\roman*)]
\item $\det(A)=0$;
\item $\mathcal{G}_A$ contains a cycle.
\end{enumerate}
\end{Theorem}

We establish the theorem below. We start by recalling a known fact~\cite{knebel2013coexistence,knebel2015evolutionary}:

\begin{Lemma}
The replicator dynamics~\eqref{eq:zsreplicator} is coexistent if and only if $K$ is nonempty. 
\end{Lemma}

It is clear that a necessary condition for $K$ to be nonempty is that $A$ has a nontrivial null space, which can hold if and only if $\det(A) = 0$.
For the remainder of the section, we assume $\det(A)=0$ and show that $K\neq\varnothing$ if and only if $\mathcal{G}_A$ contains a cycle. 
%We will show that there will be two mutually exclusive cases, namely the case where...  If $\mathcal{G}_A$ contains a directed $3$-cycle, then Lemma~\ref{lem:3cycle} shows that $L$ intersects the relative interior of a face of $\Delta^3$, which implies $K\neq\varnothing$. If $\mathcal{G}_A$ contains no directed $3$-cycle, then any cycle of $\mathcal{G}_A$ must be a directed $4$-cycle; this case is handled by Lemma~\ref{lem:twoedge4cycle}. Lemma~\ref{lem:4cycle} serves as an intermediate result for the latter case, and is also used later for describing the geometry of the equilibrium set. The relevant graph configurations are illustrated in Fig.~\ref{fig:combined_cases}.
% }
%For the remainder of the section, we assume $\det(A) = 0$ and show that  $K \neq \varnothing$ if and only if $\mathcal{G}_A$ contains a cycle.   

Since $A$ is nonzero (by the hypothesis of Theorem~\ref{thm:main2}), the dimension of the null space of $A$ is $2$. 
Recall that $H$ is the hyperplane in $\R^4$ that contains $\Delta^3$. 
Let 
$$
L:= \nll(A) \cap H.
$$
Then, either $L = \varnothing$ or $L$ is a straight line. 
In particular, if $K = L \cap \rint \Delta^3$ is nonempty, it is an open line segment. 

{  
Suppose that $K\neq\varnothing$. Then $L\cap\Delta^3$ is a line segment with two endpoints on $\partial\Delta^3$. There are two mutually exclusive cases: 

{\it Case 1:} $L$ intersects the relative interior of a certain face of $\Delta^3$. We address this case in Lemma~\ref{lem:3cycle}. 

{\it Case 2:} $L$ intersects the relative interiors of two non-incident edges of $\Delta^3$. This case will be covered by Lemmas~\ref{lem:4cycle} and~\ref{lem:twoedge4cycle}. 

These lemmas will further be used to classify the replicator dynamics~\eqref{eq:zsreplicator}, the results of which will be given in Theorem~\ref{thm:fineclass}. The corresponding configurations of $L\cap \Delta^3$ are shown in Fig.~\ref{fig:combined_cases}.}

{ For each $i \in \mathcal{V}$, we let $F_{-i}$ be the face of $\Delta^{3}$ that contains the other three vertices $e_j$, for $j\neq i$, i.e.,
$$
F_{-i} := \{x\in \Delta^3 \mid x_i = 0\}.
$$
% If the straight line $L$ intersects  $\rint F_{-i}$ transversally in $H$, for some $i\in \mathcal{V}$ (in this case, the intersection comprises a single point), then $K$ is nonempty. 
The following lemma provides a necessary and sufficient condition for
$L$ to intersect $\rint F_{-i}$ transversally in $H$.
% such intersection to occur. 
}
\begin{Lemma}\label{lem:3cycle}
The straight line $L$ intersects $\rint F_{-i}$ transversally in $H$ if and only if the subgraph of $\mathcal{G}_A$ induced by $\mathcal{V} - \{i\}$ is a $3$-cycle.
\end{Lemma}
\begin{proof}
    { Without loss of generality, we assume that $i=1$.} By the definition of $\mathcal{G}_A$, the subgraph induced by $\mathcal{V}-\{1\}$ is a $3$-cycle if and only if
    \begin{equation}
    \label{eq:cond_face}
    \sgn(a_{23})=\sgn(a_{34})=\sgn(a_{42}) \neq 0.
    \end{equation}
    We establish below the necessity and sufficiency of~\eqref{eq:cond_face} for the transversal intersection of $L$ and $\rint F_{-1}$.  

{\it{Proof of Necessity.}} If $L$ intersects $\rint F_{-1}$ transversally, then the intersection has a single point, which we denote by $z = (z_1,\ldots, z_4)$. Since $Az = 0$, we have that
\begin{equation}
\label{eq:rowz234}
\left\{
\begin{aligned}
a_{23}z_3+a_{24}z_4&=0,\quad\\
-a_{23}z_2+a_{34}z_4&=0,\quad\\
-a_{24}z_2-a_{34}z_3&=0.
\end{aligned}\right.
\end{equation}
Because $z\in \rint F_{-1}$, $z_i > 0$ for $i = 2, 3,4$. Thus, for~\eqref{eq:rowz234} to hold, either~\eqref{eq:cond_face} holds or $a_{23}=a_{24}=a_{34}=0$. 
For the latter case, we claim that $L$ cannot intersect $F_{-1}$ transversally. 
Indeed, if $a_{23}=a_{24}=a_{34}=0$, then $A$ takes the following form:
$$
A=
\begin{bmatrix}
0&a_{12}&a_{13}&a_{14}\\
-a_{12}&0&0&0\\
-a_{13}&0&0&0\\
-a_{14}&0&0&0
\end{bmatrix}.
$$
Since $A\neq 0$ (by hypothesis of Theorem~\ref{thm:main2}), any $z\in L$ satisfies $z_1 = 0$. It follows that the line $L$ lies in the hyperplane of $\R^4$ that contains $F_{-1}$. In particular, if $L$ intersects $\rint F_{-1}$, the intersection cannot be transversal.

{\it{Proof of Sufficiency.}}
We now assume that~\eqref{eq:cond_face} holds and show that there exists a unique $z\in L \cap \rint F_{-1}$.  
%Suppose the subgraph induced by $\mathcal{V}-\{i\}$ is a $3$-cycle, then we have~\eqref{eq:cond_face} holds. 
Consider the vector
$v:= (0,a_{34},-a_{24},a_{23})$. By~\eqref{eq:cond_face}, $v_2$, $v_3$, and $v_4$ are nonzero and have the same sign, so 
$\mathbf{1}^\top v\neq 0$. We then define
$z:=\frac{1}{\mathbf{1}^\top v}\,v$. 
By construction, we have that $z_1=0$, $z_i > 0$ for $i = 2,3,4$, and $\mathbf{1}^\top z=1$. Thus, $z\in \rint F_{-1}$. Also, it follows from computation that 
$$
Az
%=
% \frac{1}{\mathbf{1}^\top v}
% \begin{bmatrix}
% a_{12}a_{34}-a_{13}a_{24}+a_{14}a_{23}\\
% a_{23}(-a_{24})+a_{24}(a_{23})\\
% (-a_{23})a_{34}+a_{34}(a_{23})\\
% (-a_{24})a_{34}+(-a_{34})(-a_{24})
% \end{bmatrix}
=\frac{1}{\mathbf{1}^\top v}
\begin{bmatrix}
a_{12}a_{34}-a_{13}a_{24}+a_{14}a_{23}\\
0\\
0\\
0
\end{bmatrix} = 0,
$$
where the last equality follows from the fact that  
$\det(A)= (a_{12}a_{34}-a_{13}a_{24}+a_{14}a_{23})^2=0$. Thus, $z\in L\cap \rint F_{-1}$. 
It now remains to show that $z$ is the only point in $L\cap \rint F_{-1}$. Let $z'=(0,z'_2,z'_3,z'_4)\in L \cap \rint F_{-1}$.
Since $Az'=0$, we have that
\begin{equation}\label{eq:rows234}
\left\{
\begin{aligned}
 a_{23}z'_3+a_{24}z'_4&=0,\\
 -a_{23}z'_2+a_{34}z'_4&=0,\\
 -a_{24}z'_2-a_{34}z'_3&=0.
\end{aligned}
\right. 
\end{equation}
On the one hand, it follows from~\eqref{eq:cond_face} that a solution of the homogeneous equation~\eqref{eq:rows234} must be a scalar multiple of $(z_2, z_3, z_4)$, i.e., $(z'_2,z'_3,z'_4) = t(z_2, z_3, z_4)$ for some $t \in \R$. On the other hand, for $z'\in \rint F_{-1}$ to hold, we must have that $\sum_{i = 2}^4 z'_i = 1$. We thus conclude that $t = 1$ and hence, $z' = z$. 
This completes the proof. 
\end{proof}
% is one-dimensional. In fact, it forces
% $$
% (z'_2,z'_3,z'_4)=t\,(a_{34},-a_{24},a_{23})
% \quad\text{for some }t>0.
% $$
% Normalizing such that $\mathbf{1}^\top z'=1$ yields $z'=z$. Hence,
% $\nll(A)\cap\rint(F_{-1})=\{z\}$.

Besides the scenario where $L$ intersects $\rint F_{-i}$, for some $i\in \mathcal{V}$, at a (unique) point, the other scenario that can lead to nonempty $K$ is that $L$ intersects the interiors of two non-incident edges of $\Delta^3$. We address below the latter scenario. 
For any two distinct $i,j\in \mathcal{V}$, let 
$$E_{ij}:= \{ x\in \Delta^3 \mid x_k = 0, \mbox{ for } k\in \mathcal{V} - \{i,j\}\}.$$ 
%{  ///same problem, add transversally.///  We say that $L$ intersects $\rint(E_{ij})$ transversally if
%$L$ intersects $\rint(E_{ij})$ at a single point and  $L\cap \rint \Delta^3\neq \varnothing$.}
%Then, the aforementioned $L$ satisfies the condition that there exist two edges $E_{ij}$ and $E_{k\ell}$, with $\mathcal{V} = \{i,j,k,\ell\}$, such that 
%\begin{equation}\label{eq:Lintersectsedges}
%L \cap \rint E_{ij} \neq \varnothing\quad \mbox{and} \quad L \cap %\rint E_{k\ell} \neq \varnothing.
%\end{equation} 
%Conversely, if $L$ satisfies~\eqref{eq:Lintersectsedges}, then $L$ %intersects $\rint \Delta^3$.  
We need the following lemma: 
%provides a necessary and sufficient condition for $L$ to intersect $\rint E_{ij}$ transversally. 

% {  The set $L$ intersects $\rint E_{ij}$ transversally if and only if $a_{ij}=0$, and there exists a $r>0$, such that  $$ra_{ik}+a_{jk}=0\mbox{ and }a_{ik}a_{jk}\neq 0$$ for all $k\in \mathcal{V}-\{i,j\}$.
% }
\begin{Lemma}\label{lem:4cycle}
The set $L$ intersects both $\rint E_{ij}$ and $\rint \Delta^3$ if and only if
$a_{ij}=0$ and $\mathcal{G}_A$ contains a $4$-cycle.
\end{Lemma}

\begin{proof}
Without loss of generality, we assume that $i=1$ and $j=2$.
We can parameterize $\rint E_{12}$ as
$$
\rint E_{12} = \left \{y(r):=\frac{1}{1+r}(r,1,0,0) \mid  r>0 \right \}.
$$
%We first characterize when $z(s)\in L$. A direct computation yields
% \begin{equation}
% \label{eq:edge_eqs_4cycle}
% A z(r)=\frac{1}{1+r}
% \begin{bmatrix}
% a_{12}\\
% -r a_{12}\\
% -r a_{13}-a_{23}\\
% -r a_{14}-a_{24}
% \end{bmatrix}. 
% \end{equation}
It follows from computation that $A y(r) = 0$ (i.e., $L$ intersects $\rint E_{12}$) if and only if 
\begin{equation}\label{eq:edge_kernel_condition}
a_{12}=0,\qquad ra_{13}+a_{23}=0,\qquad ra_{14}+a_{24}=0.
\end{equation}
We now prove the two implications.

We first assume that~\eqref{eq:edge_kernel_condition} holds for some $r > 0$ and $z=(z_1,z_2,z_3,z_4)\in L\cap \rint\Delta^3$, and prove that $a_{12} = 0$ and $\mathcal{G}_A$ has a $4$-cycle.  
% By~\eqref{eq:edge_kernel_condition}, we have
% \begin{equation}\label{eq:edge_relations}
% a_{12}=0,\qquad ra_{13}+a_{23}=0,\qquad ra_{14}+a_{24}=0.
% \end{equation} 
Since $z\in L$, the first row of $Az=0$ yields
$
a_{12}z_2+a_{13}z_3+a_{14}z_4=0$. 
Using $a_{12}=0$ from~\eqref{eq:edge_kernel_condition}, we obtain that
\begin{equation}\label{eq:kernel_2}
a_{13}z_3+a_{14}z_4=0.
\end{equation}
By~\eqref{eq:kernel_2} and the fact that $z\in \rint \Delta^3$, we have that either $\sgn(a_{13}) = -\sgn(a_{14}) \neq 0$ or
$a_{13}=a_{14}=0$. But, if the latter holds, then by~\eqref{eq:edge_kernel_condition}, $a_{23} = a_{24} = 0$. Further, by the fact that $Az = 0$ with $z\in \rint \Delta^3$, we have that $A = 0$, which is a contradiction. We thus must have that $\sgn(a_{13}) = -\sgn(a_{14}) \neq 0$.   
Combining this fact with~\eqref{eq:edge_kernel_condition}, we have that
\begin{equation}\label{eq:4cycle_sign}
\sgn(a_{13})=\sgn(a_{32})=\sgn(a_{24})=\sgn(a_{41})\neq 0,
\end{equation}
so $\mathcal G_A$ contains a $4$-cycle.
% We further obtain
% $$
% \sgn(a_{13})=\sgn(a_{32})\mbox{ and }
% \sgn(a_{14})=\sgn(a_{42}),
% $$
% which follows directly from~\eqref{eq:ratio} and the skew-symmetricity of $A$.

Conversely, we assume that $a_{12}=0$ and $\mathcal G_A$ contains a $4$-cycle, and show that~\eqref{eq:edge_kernel_condition} holds for some $r > 0$ and that $L$ intersects $\rint \Delta^3$. 
Since $a_{12} = 0$, there is no edge between nodes $1$ and $2$. Then, it is not hard to see that $\mathcal{G}_A$ has a $4$-cycle if and only if~\eqref{eq:4cycle_sign} holds.
Further, note that
\begin{equation*}
\det(A) =(a_{12}a_{34}-a_{13}a_{24}+a_{14}a_{23})^2 
=(-a_{13}a_{24}+a_{14}a_{23})^2=0.
\end{equation*}
Since $a_{13}\neq 0$ and $a_{14} \neq 0$, we define
{  
$$
r:=-\frac{a_{23}}{a_{13}}=-\frac{a_{24}}{a_{14}}.
$$
By~\eqref{eq:4cycle_sign}, we have that $r > 0$ and hence, ~\eqref{eq:edge_kernel_condition} holds.}
% where last inequality follows from~\eqref{eq:4cycle_sign}. we have $r>0$, and 
% $$a_{12}=0,\qquad ra_{13}+a_{23}=0,\qquad ra_{14}+a_{24}=0.$$
% By condition
% \eqref{eq:edge_kernel_condition}, we have
% $$
% z(r)=\frac{1}{1+r}(r,1,0,0)\in L\cap \rint E_{12}.
% $$
Finally, we define
{ $$
v:=a_{13}\begin{bmatrix}
a_{34}\\0\\-a_{14}\\a_{13}
\end{bmatrix}.
$$}
By \eqref{eq:4cycle_sign}, $v_3 > 0$ and $v_4 > 0$. Also, it follows directly from computation that $Av=0$, so $v\in \nll(A)$. 
Thus, for any sufficiently small $\varepsilon>0$, the point 
$u:=y(r)+\varepsilon v$ has all entries positive, and moreover, satisfies $A u = 0$. We conclude that $z:= u/(\mathbf 1^\top u)$ belongs to $L\cap \rint \Delta^3$. 
This completes our proof.
\end{proof}
% has strictly positive coordinates. After normalization, we have
% $$
% z:=\frac{u}{\mathbf 1^\top u}\in L\cap \rint \Delta^3.
% $$
% Therefore,
% $$
% L\cap \rint E_{12}\neq \varnothing
% \qquad\text{and}\qquad
% L\cap \rint\Delta^3\neq \varnothing .
% $$

An immediate consequence of Lemma~\ref{lem:4cycle} is the following result:

\begin{Lemma}
\label{lem:twoedge4cycle}
    There exist two non-incident edges $E_{ij}$ and $E_{k\ell}$ of $\Delta^3$ such that the set $L$ intersects both $\rint E_{ij}$ and $\rint E_{k\ell}$ if and only if $\mathcal{G}_A$ is a $4$-cycle.
\end{Lemma}

\begin{proof}
Without loss of generality, we set $i=1,j=2,k=3,\ell=4$.
Assume that
$L\cap \rint E_{12}\neq \varnothing$ and $L\cap \rint E_{34}\neq \varnothing$, we prove $\mathcal{G}_A$ is a $4$-cycle. Pick
$y\in L\cap \rint E_{12}$, and
$y'\in L\cap \rint E_{34}$. Since $E_{12}$ and $E_{34}$ are non-incident, the open line segment joining $y$ and $y'$
lies in $\rint \Delta^3$. Hence, 
$L\cap \rint \Delta^3\neq \varnothing$. 
Applying Lemma~\ref{lem:4cycle} to $E_{12}$ and $E_{34}$, we obtain that
$a_{12}=0$, $a_{34}=0$,  
and $\mathcal{G}_A$ contains a $4$-cycle. Since $\{1,2\}\cap\{3,4\}=\varnothing$,
$(1,2)$ and $(3,4)$ are precisely the two missing edges on four vertices.
Therefore, $\mathcal{G}_A$ itself is a $4$-cycle. 

Conversely, we assume that $\mathcal G_A$ is a $4$-cycle and show that $L$ intersects two non-incident edges. By the definition of $\mathcal{G}_A$, we can write (up to relabeling)
\begin{equation*} \sgn(a_{13})=\sgn(a_{32})=\sgn(a_{24})=\sgn(a_{41})\neq 0, \mbox{ and } 
a_{12}=0, a_{34}=0.
\end{equation*}
By Lemma~\ref{lem:4cycle}, we have that
$L\cap \rint E_{12}\neq \varnothing$
 and 
$L\cap \rint E_{34}\neq \varnothing$. 
\end{proof}

To complete the proof of Theorem~\ref{thm:main2}, it suffices to establish the following lemma:
\begin{Lemma}
If $\mathcal{G}_A$ has a cycle but does not have a $3$-cycle, then $\mathcal{G}_A$ itself must be a $4$-cycle.
\end{Lemma}
\begin{proof}
{  
Let $\mathcal{C}$ be a cycle of $\mathcal{G}_A$. As noted right before Theorem~\ref{thm:main2}, $\mathcal{G}_A$ has only $3$- and $4$-cycles. By the hypothesis of the lemma, $\mathcal{G}_A$ does not have a $3$-cycle, so $\mathcal{C}$ must be a $4$-cycle.}
%Notice that $\mathcal{G}_A$ is composed of $4$ vertices and contains no $3$-cycle, then $\mathcal{G}_A$ must contain a $4$-cycle. 
Without loss of generality, we express $\mathcal{C}$ (up to permutation) as 
$\mathcal{C}=(1,2,3,4,1)$. The only possible additional edges are between the pairs $(1,3)$ and $(2,4)$. We show that any such edge creates a directed $3$-cycle. Consider $4$ cases:
\begin{description}
    \item[\it{Case 1.}] If $(1,3)\in\mathcal{E}$, then $(1,3,4,1)$ is a $3$-cycle;\\
    \item[{\it{Case 2.}}]  If $(3,1)\in\mathcal{E}$, then $(1,2,3,1)$ is a $3$-cycle;\\
    \item[{\it{Case 3.}}]  If $(2,4)\in\mathcal{E}$, then $(1,2,4,1)$ is a $3$-cycle;\\
    \item[{\it{Case 4.}}]  If $(4,2)\in\mathcal{E}$, then $(2,3,4,2)$ is a $3$-cycle.
\end{description}
Hence no additional edge can be present, and $\mathcal{G}_A$ is itself a $4$-cycle.
\end{proof}

\subsection{A finer classification}
\label{ssec:finer_class}

\begin{figure*}[t]
\centering
\setlength{\tabcolsep}{3pt}

\begin{tabular}{ccccc}

%==================== Top row ====================

\begin{subfigure}[t]{0.19\textwidth}
\centering
\begin{tikzpicture}[baseline=(current bounding box.center), scale=2.2,
  v/.style={circle, fill=black, inner sep=1.5pt},
  e/.style={line width=0.8pt, <-, >=Stealth}
]
  \node[v, label=below:$1$]  (n1) at (0,0) {};
  \node[v, label=below:$2$] (n2) at (1,0) {};
  \node[v, label=above:$3$]       (n3) at (0.5,{sqrt(3)/2}) {};
  \node[v, label=below:$4$]       (n4) at (0.5,{sqrt(3)/6}) {};
  \draw[e] (n2) -- (n3);
  \draw[e] (n3) -- (n4);
  \draw[e] (n4) -- (n2);
  \draw[e] (n4) -- (n1);
  \draw[e] (n1) -- (n3);
  \draw[e] (n1) -- (n2);
\end{tikzpicture}
\phantomsubcaption
\makeatletter
\def\@currentlabel{I}
\makeatother
\label{subfig:I}

\caption*{(I)}
\end{subfigure}
&
\begin{subfigure}[t]{0.19\textwidth}
\centering
\begin{tikzpicture}[baseline=(current bounding box.center), scale=2.2,
  v/.style={circle, fill=black, inner sep=1.5pt},
  e/.style={line width=0.8pt, <-, >=Stealth}
]
  \node[v, label=below:$1$]  (n1) at (0,0) {};
  \node[v, label=below:$2$] (n2) at (1,0) {};
  \node[v, label=above:$3$]       (n3) at (0.5,{sqrt(3)/2}) {};
  \node[v, label=below:$4$]       (n4) at (0.5,{sqrt(3)/6}) {};
  \draw[e] (n2) -- (n3);
  \draw[e] (n3) -- (n4);
  \draw[e] (n4) -- (n2);
  \draw[e] (n4) -- (n1);
  \draw[e] (n1) -- (n3);
\end{tikzpicture}
\phantomsubcaption
\makeatletter
\def\@currentlabel{II}
\makeatother
\label{subfig:II}

\caption*{(II)}\label{subfigure:II}
\end{subfigure}
&
\begin{subfigure}[t]{0.19\textwidth}
\centering
\begin{tikzpicture}[baseline=(current bounding box.center), scale=2.2,
  v/.style={circle, fill=black, inner sep=1.5pt},
  e/.style={line width=0.8pt, <-, >=Stealth}
]
  \node[v, label=below:$1$]  (n1) at (0,0) {};
  \node[v, label=below:$2$] (n2) at (1,0) {};
  \node[v, label=above:$3$]       (n3) at (0.5,{sqrt(3)/2}) {};
  \node[v, label=below:$4$]       (n4) at (0.5,{sqrt(3)/6}) {};
  \draw[e] (n2) -- (n4);
  \draw[e] (n4) -- (n3);
  \draw[e] (n3) -- (n2);
  \draw[e] (n1) -- (n3);
  \draw[e] (n4) -- (n1);
\end{tikzpicture}
\phantomsubcaption
\makeatletter
\def\@currentlabel{III}
\makeatother
\label{subfig:III}

\caption*{(III)}
\end{subfigure}

&
\begin{subfigure}[t]{0.19\textwidth}
\centering
\begin{tikzpicture}[baseline=(current bounding box.center), scale=2.2,
  v/.style={circle, fill=black, inner sep=1.5pt},
  e/.style={line width=0.8pt, <-, >=Stealth}
]
  \node[v, label=below:$1$]  (n1) at (0,0) {};
  \node[v, label=below:$2$] (n2) at (1,0) {};
  \node[v, label=above:$3$]       (n3) at (0.5,{sqrt(3)/2}) {};
  \node[v, label=below:$4$]       (n4) at (0.5,{sqrt(3)/6}) {};
  \draw[e] (n2) -- (n4);
  \draw[e] (n4) -- (n3);
  \draw[e] (n3) -- (n2);
\end{tikzpicture}
\phantomsubcaption
\makeatletter
\def\@currentlabel{IV}
\makeatother
\label{subfig:IV}

\caption*{(IV)}
\end{subfigure}
&
\begin{subfigure}[t]{0.19\textwidth}
\centering
\begin{tikzpicture}[baseline=(current bounding box.center), scale=2.2,
  v/.style={circle, fill=black, inner sep=1.5pt},
  e/.style={line width=0.8pt, <-, >=Stealth}
]
  \node[v, label=below:$1$]  (n1) at (0,0) {};
  \node[v, label=below:$2$] (n2) at (1,0) {};
  \node[v, label=above:$3$]       (n3) at (0.5,{sqrt(3)/2}) {};
  \node[v, label=below:$4$]       (n4) at (0.5,{sqrt(3)/6}) {};
  \draw[e] (n1) -- (n3);
  \draw[e] (n4) -- (n1);
  \draw[e] (n3) -- (n2);
  \draw[e] (n2) -- (n4);
\end{tikzpicture}
\phantomsubcaption
\makeatletter
\def\@currentlabel{V}
\makeatother
\label{subfig:V}

\caption*{(V)}
\end{subfigure}

\\[3mm]

%==================== Bottom row ====================

\begin{subfigure}[t]{0.18\textwidth}
\centering
\includegraphics[width=1.15\linewidth]{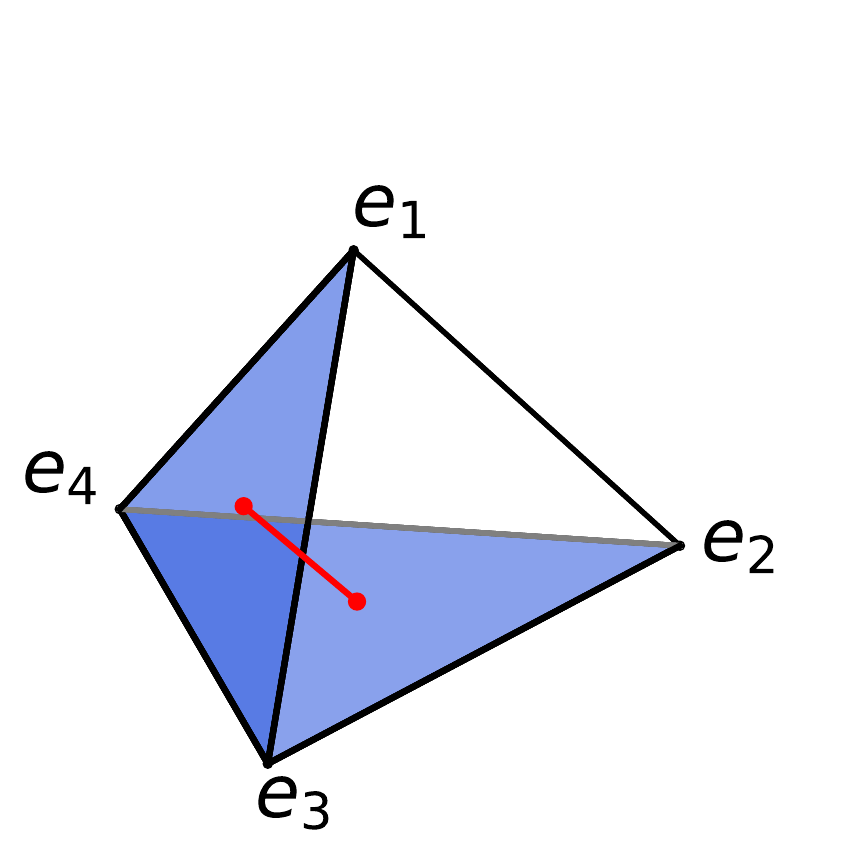}
\caption*{}
\end{subfigure}
\phantomsubcaption\label{subfig:2a}
&
\begin{subfigure}[t]{0.18\textwidth}
\centering
\includegraphics[width=1.15\linewidth]{Figures/inter_case_a.pdf}
\caption*{}
\end{subfigure}
\phantomsubcaption\label{subfig:2b}
&
\begin{subfigure}[t]{0.18\textwidth}
\centering
\includegraphics[width=1.15\linewidth]{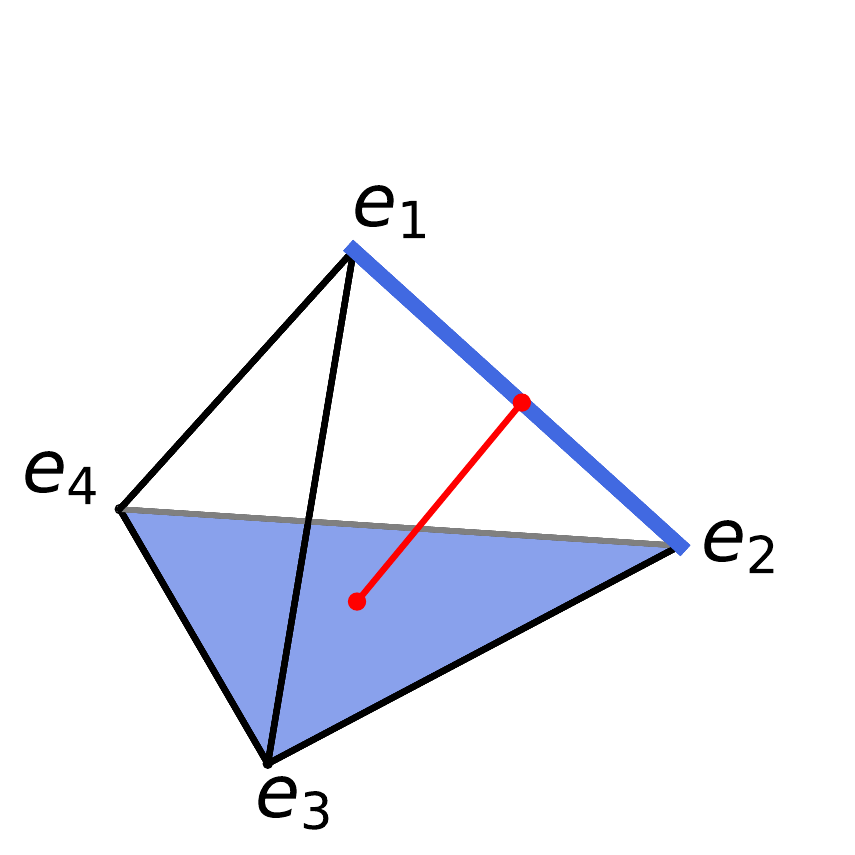}
\caption*{}
\end{subfigure}
\phantomsubcaption\label{subfig:2c}
&
\begin{subfigure}[t]{0.18\textwidth}
\centering
\includegraphics[width=1.15\linewidth]{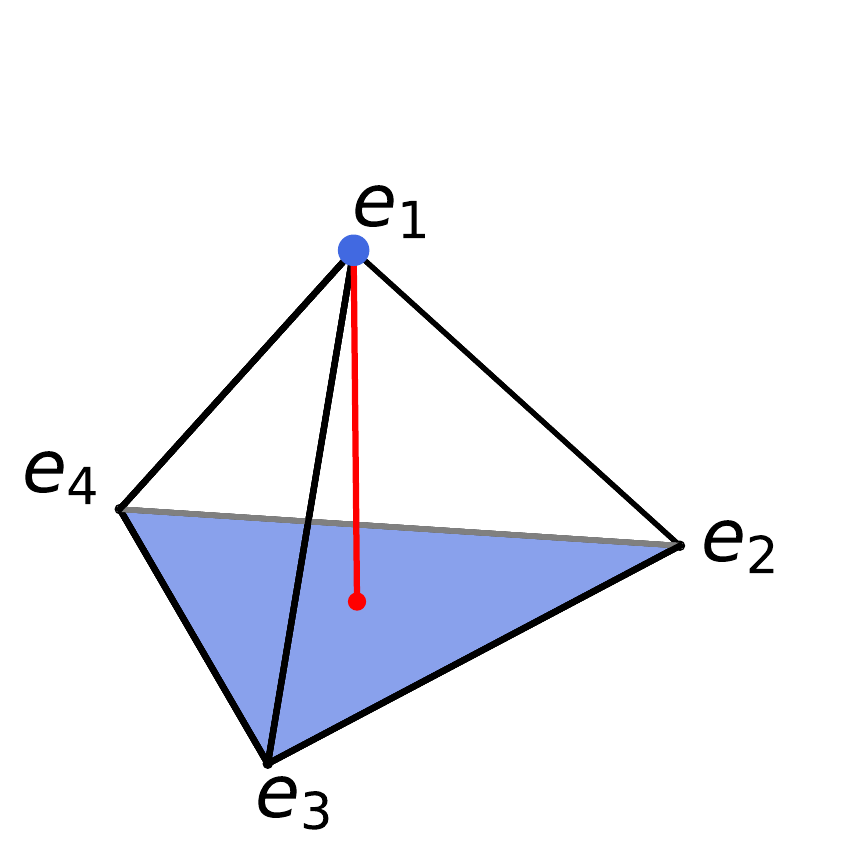}
\caption*{}
\end{subfigure}
\phantomsubcaption\label{subfig:2d}
&
\begin{subfigure}[t]{0.18\textwidth}
\centering
\includegraphics[width=1.15\linewidth]{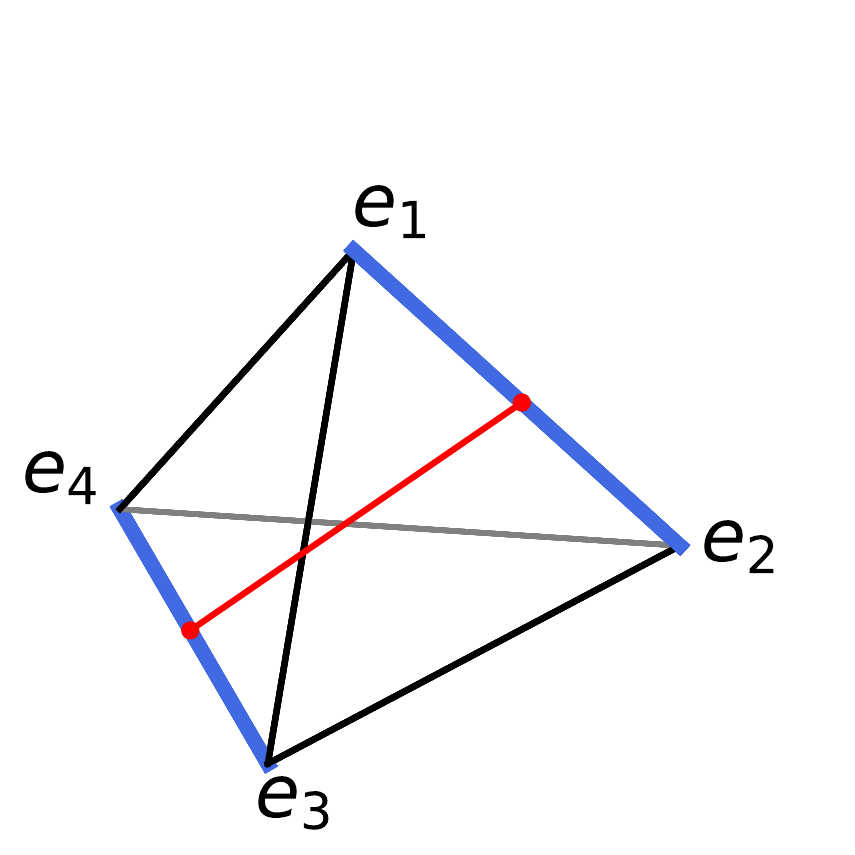}
\caption*{}
\end{subfigure}
\phantomsubcaption\label{subfig:2e}

\end{tabular}

\caption{
The five isomorphism classes of $\mathcal{G}_A$ that contains cycle. The top row shows the graph $\mathcal{G}_A$. 
Among these,~(\ref{subfig:I}) and~(\ref{subfig:II}) each contain two directed $3$-cycles;~(\ref{subfig:III}) contains both a directed $3$-cycle and a directed $4$-cycle;~(\ref{subfig:IV}) contains exactly one directed $3$-cycle; and~(\ref{subfig:V}) is a directed $4$-cycle. 
The bottom row shows the corresponding set $L$ in $\Delta^3$, matched columnwise.}
\label{fig:combined_cases}
\end{figure*}

% {\color{red}{increase distance of $e_4$; }}

Up to isomorphism, there are five different digraphs $\mathcal{G}_A$ that can contain a cycle. We present these graphs in Fig.~\ref{fig:combined_cases}. 
To every such digraph, we characterize below how the straight line $L$ intersects $\Delta^3$ (under the condition $\det(A) = 0$).  

% { {fix the graph}}

% {\it{Case $1(a)$: Two endpoints lie on the interior of two different faces.}} We have the following lemma.

\begin{Theorem}
\label{thm:fineclass}
Given a skew-symmetric $A$ with $\det(A)=0$, and let $\mathcal{G}_A$ be the associated digraph. Then, the following hold:
\begin{enumerate}
    \item If $\mathcal{G}_A$ is isomorphic to the digraph~(\ref{subfig:I}) or~(\ref{subfig:II}) in Fig.~\ref{fig:combined_cases}, then there exist two distinct $i,j\in \mathcal{V}$ such that the two endpoints of $L$ belong to $\rint F_{-i}$ and $\rint F_{-j}$ respectively. 
    \item 
    If $\mathcal{G}_A$ is isomorphic to the digraph~(\ref{subfig:III}) in Fig.~\ref{fig:combined_cases}, then there exist two distinct $i,j\in \mathcal{V}$ such that the two endpoints of $L$ belong to $\rint F_{-i}$ and $\rint E_{ij}$ respectively. 
    \item If $\mathcal{G}_A$ is isomorphic to the digraph~(\ref{subfig:IV}) in Fig.~\ref{fig:combined_cases}, then there exists an $i\in \mathcal{V}$ such that the two endpoints of $L$ belong to $\rint F_{-i}$ and $\{e_i\}$ respectively. 
    \item If $\mathcal{G}_A$ is isomorphic to the digraph~(\ref{subfig:V}) in Fig.~\ref{fig:combined_cases}, then there exist two non-incident edges $E_{ij}$ and $E_{k\ell}$ such that the two endpoints of $L$ belong to $\rint E_{ij}$ and $\rint E_{k\ell}$ respectively. 
\end{enumerate}

\end{Theorem}

\begin{proof}
Theorem~\ref{thm:fineclass} follows directly from Lemma~\ref{lem:3cycle}, \ref{lem:4cycle}, and \ref{lem:twoedge4cycle}. Indeed, the graphs~(\ref{subfig:I}) and~(\ref{subfig:II}) in Fig.~\ref{fig:combined_cases} each contains two directed $3$-cycles. Hence, by Lemma~\ref{lem:3cycle}, the set $L$ intersects transversally the relative interiors of two distinct faces. {  (Graph~(\ref{subfig:I}) also contains a directed $4$-cycle. However, since $a_{12}\neq 0$, Lemma~\ref{lem:4cycle} does not apply.)}

%The graph~(\ref{subfig:III}) in Fig.~\ref{fig:combined_cases} contains a directed $3$-cycle. In addition, unlike graph~(\ref{subfig:I}), graph~(\ref{subfig:III}) contains a directed $4$-cycle together with a missing edge between two vertices of the cycle. }Therefore, Lemma~\ref{lem:3cycle} yields an intersection of $L$ with the relative interior of a face, while Lemma~\ref{lem:4cycle} yields an intersection of $L$ with the relative interior of an edge. Moreover, $L\cap \rint \Delta^3\neq\varnothing$.
The graph~(\ref{subfig:III}) contains both a directed $3$-cycle and a $4$-cycle. Therefore, Lemma~\ref{lem:3cycle} yields an intersection of $L$ with the relative interior of a face, while Lemma~\ref{lem:4cycle} yields an intersection of $L$ with the relative interior of an edge, and moreover, $L\cap \rint \Delta^3\neq\varnothing$. 

The graph~(\ref{subfig:IV}) contains exactly one directed $3$-cycle and no directed $4$-cycle. Thus, by Lemma~\ref{lem:3cycle}, the set $L$ intersects transversally the relative interior of a unique face, whereas Lemmas~\ref{lem:4cycle} and \ref{lem:twoedge4cycle} do not apply. Consequently, the other endpoint of $K$ must be the vertex that does not belong to the face. 

Finally, the graph~(\ref{subfig:V}) is exactly a directed $4$-cycle. Hence, by Lemma~\ref{lem:twoedge4cycle}, the set $L$ intersects the relative interiors of two non-incident edges. This proves all four items.
\end{proof}

\section{System Behavior in the Interior}
\label{sec:sys_interior}
In this section, we assume that the skew-symmetric matrix $A$ satisfies the two items of Theorem~\ref{thm:main2}, so the set $K$ is nonempty. It is clear that $K$ comprises all the equilibria of~\eqref{eq:zsreplicator} in $\rint \Delta^3$. 
It has been shown that each equilibrium point $z\in K$ is stable in the Lyapunov sense, but not asymptotically stable (see, e.g.,~\cite[Theorem~4]{AkinLosert1984} for detail); {  in fact, using the same arguments of the paper, one can show that the two endpoints of $K$ are also stable, but not asymptotically stable. Note that the closure of $K$ is $L\cap \Delta^3$. We thus record the following result:

% {\color{red} find the exact theorem to refer to}

\begin{Lemma}
\label{lem:L_stability}
    Every point of $L\cap \Delta^3$ is stable but not asymptotically stable.
\end{Lemma}
}

Furthermore, it is known~\cite{BoonePiliouras2019WINE} 
that the dynamics~\eqref{eq:zsreplicator} is recurrent in $\rint \Delta^3$ (note that $\rint \Delta^3$ is forward invariant), i.e., for any $x(0) \in \rint \Delta^3$ and for any $\epsilon > 0$, there exists a time instant $t$ such that $\|x(t) - x(0) \|< \epsilon$.  

We strengthen the above recurrence statement by showing that 
every point $x(0)\in \rint\Delta^3 - K$ belongs to a periodic orbit: 

\begin{Theorem}
\label{thm:periodic_orbit}
    Every point in the set $\rint \Delta^3 - K$ belongs to a periodic orbit of~\eqref{eq:zsreplicator}. Moreover, any such periodic orbit is stable. 
\end{Theorem}

\begin{Remark}\normalfont
    {  %We provide a geometric interpretation for the periodic orbit.  
    %A geometric interpretation of $K$ has been provided in the literature: %It is well known in the literature  that 
    It is well known (see, e.g., Section~2.3 of~\cite{HofbauerSigmund2003}) that the time-average (over a period) of a periodic orbit is an equilibrium point. Moreover, if the orbit is in the interior of the simplex, then the associated equilibrium point is also in the interior and is a Nash equilibrium of the population game induced by the payoff matrix/function. 
    In the present setting, the time average of any interior periodic orbit  of~\eqref{eq:zsreplicator} belongs to~$K$.}
    More precisely, let $\gamma$ be a periodic orbit in $\rint \Delta^3$ with period $T_\gamma$. Let $x(0)\in \gamma$ and we define the time-average of the state:
    $$
    \bar x_\gamma:=\frac{1}{T_\gamma} \int_0^{T_\gamma} x(t) \mathrm{d}t.
    $$
    It should be clear that $\bar x_\gamma$ depends only on $\gamma$, but not on a specific choice of $x(0)$. We have $\bar x_\gamma \in K$; see Fig.~\ref{fig:orbit} for illustration.% It is known~\cite{hofbauer2003evolutionary} that $\bar x_\gamma \in K$. 
    %{  This gives a geometric interpretation of $K$: although the trajectory evolves along a periodic orbit in $\rint\Delta^3-K$, its time average lies on the equilibrium continuum $K$.}
\end{Remark}

\begin{figure}[h]
    \centering
    \includegraphics[width=0.6\linewidth, trim={2cm 1cm 5.5cm 6cm}, clip]{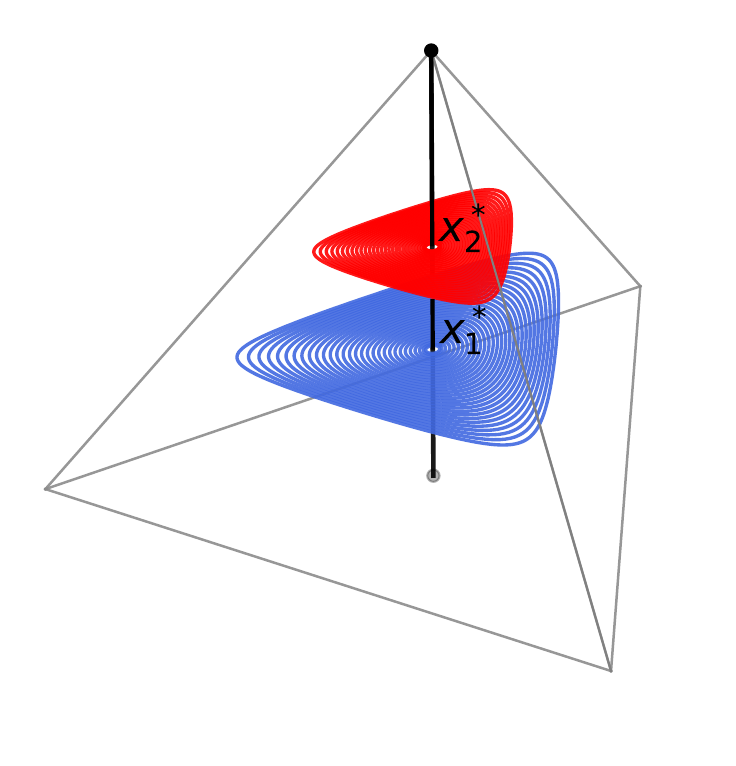}
    \caption{{ Periodic orbits and their time averages for the replicator dynamics induced by a payoff matrix $A$ whose digraph is isomorphic to~(\ref{subfig:IV}) in Fig.~\ref{fig:combined_cases}. 
    The time average of each blue (resp., red) periodic orbit is $x_1^*$ (resp., $x_2^*$), where $x_1^*$ and $x_2^*$ are two equilibria in the relative interior of the simplex belonging to the segment $K$ plotted in black.}}
    \label{fig:orbit}
\end{figure}

% \subsection{Stability of the Equilibrium Continuum}
% \label{subsec:stability_K}

% In this subsection, we establish the following result:

% \begin{Theorem}\label{thm:neutral_stability}
% %Let matrix $A$ be defined as in~\eqref{eq:entriesofA}. 
% % and consider the replicator dynamics \eqref{eq:zsreplicator}.
% Each equilibrium point $z\in K$ is stable, but not asymptotically stable.
% \end{Theorem}

% To prove the theorem, 
{  We provide here a sketch of the proof for Theorem~\ref{thm:periodic_orbit}. 
To this end, we will introduce a family of strictly convex functions $\phi_{z}(x)$ (see~\eqref{eq:def_phi} below), parameterized by $z\in K$. 
These functions turn out to be invariant under the dynamics~\eqref{eq:zsreplicator}. 
Then, we will show that for any $x(0)\in \rint \Delta^3 - K$, there exist two distinct points $z',z''\in K$ such that the intersection of the two iso-level sets $\phi^{-1}_{z'}(\phi_{z'}(x(0)))$ and $ \phi^{-1}_{z''}(\phi_{z''}(x(0)))$, which we denote by $\Gamma_{z',z''}(x(0))$, does not intersect $K$. Moreover, we will show that $\Gamma_{z',z''}(x(0))$ is a compact, smooth one-dimensional manifold without boundary. 
It thus follows that the connected component of $\Gamma_{z',z''}(x(0))$ containing $x(0)$ is diffeomorphic to $S^1$ and has to be a periodic orbit.
Finally, we will use the invariant functions $\phi_{z'}(x)$ and $\phi_{z''}(x)$ to construct a Lyapunov function whereby we establish the stability of the orbit. 
}

% The proof uses two conserved quantities, $\phi_{z'}$ and $\phi_{z''}$, whose properties are summarized in Lemma~\ref{lem:phi_z}. We confine the trajectory from $x(0)\in\rint\Delta^3-K$ to the common level set $\Gamma_{z',z''}(x(0))$, which is the intersection of $\phi^{-1}_{z'}(\phi_{z'}(x(0)))$ and $ \phi^{-1}_{z''}(\phi_{z''}(x(0)))$. 
% We show in Lemma~\ref{lem:two_z} that $z',z''\in K$ can be chosen such that $\Gamma_{z',z''}(x(0))$ is disjoint from $K$, and then prove in Lemma~\ref{lem:regular_value} that it is a compact smooth one-dimensional submanifold of $\rint\Delta^3$. Hence,  the connected component $\gamma$ containing $x(0)$ is diffeomorphic to $S^1$. 
% Since $\gamma\cap K=\varnothing$, the vector field has no equilibrium on $\gamma$, and the induced flow is periodic. 
% Finally, the stability of $\gamma$ follows from the same conserved quantities, which yield an invariant function $V$ such that, for sufficiently small $\rho>0$, 
% the sublevel set $\{x:V(x)<\rho\}$ is a forward-invariant neighborhood of $\gamma$. The full proof of Theorem~\ref{thm:periodic_orbit} is provided at the end of this section.

We now introduce the family of functions $\phi_z : \rint\Delta^3 \to \R$, parameterized by $z\in K$ as
\begin{equation}\label{eq:def_phi}
\phi_z(x) := -\sum_{i=1}^4 z_i \ln\frac{x_i}{z_i}.
\end{equation}
We gather below a few relevant properties of the function $\phi_z$ (see, e.g.,~\cite{HofbauerSigmund2003}).

%and then use them to characterize the connected components of the level sets of $F_{z',z''}(x)$.
\begin{Lemma}
\label{lem:phi_z}
    The following items hold:
    \begin{enumerate}
        \item The function $\phi_z(x)$ is strictly convex, with $x = z$ the global minimum point and $\phi_z(z) = 0$. Moreover, $\phi_z(x)$ goes to $+\infty$ as $x$ approaches  $\partial \Delta^3$.
        
        \item The function $\phi_z(x)$ is invariant under the dynamics~\eqref{eq:zsreplicator}, i.e., $\phi_z(x(t))=\phi_z(x(0))$, for any $x(0)\in \rint\Delta^3$ and for any $t\geq 0$.
    \end{enumerate}
\end{Lemma}

For any two distinct points $z',z''\in K$, we define $F_{z',z''}:\rint \Delta^3 \to \R^2$ as follows: 
$$
F_{z',z''}(x):=\bigl(\phi_{z'}(x),\phi_{z''}(x)\bigr).
$$
By item~2 of Lemma~\ref{lem:phi_z}, $F_{z',z''}$ is forward invariant along the dynamics~\eqref{eq:zsreplicator}, so $F_{z',z''}(x(t)) = F_{z',z''}(x(0))$ for all $t\geq 0$. 
% Given a constant $c > 0$, we define the sub-level set of $\phi_z$ and its boundary as:
% $$
% \phi_z^{-1}([0,c]) := \{x \in \rint\Delta^3 : \phi_z(x) \le c\}
% \mbox{ and }
% \phi_z^{-1}(c) := \{x \in \rint\Delta^3 : \phi_z(x) = c\}.
% $$
%For ease of notation, 
We consider the preimage: 
{ \begin{equation*}
\Gamma_{z',z''}(x(0))  :=F_{z',z''}^{-1}(F_{z',z''}(x(0))) 
 = \phi^{-1}_{z'}(\phi_{z'}(x(0))) \cap \phi^{-1}_{z''}(\phi_{z''}(x(0))).
\end{equation*}}
% {  where we denote $\phi_{z}^{-1}\left(\phi_{z}(x(0))\right)$ as $\phi_{z}^{-1}(x(0))$ with slight abuse of notation.} 
The following lemma shows that $z'$ and $z''$ can be chosen so that
$\Gamma$ is disjoint from $K$.

% We next show that, for any $x(0)\in\rint\Delta^3-K$, one can always choose two equilibria $z',z''\in K$ such that the intersection of the two level sets has no intersection with $K$. Specifically, we have

\begin{Lemma}
\label{lem:two_z}
    For any given $x(0)\in \rint \Delta^3-K$, there exist two distinct points $z', z''\in K$ such that $\Gamma_{z',z''}(x(0))$ does not intersect $K$.
\end{Lemma}
\begin{proof}
    Pick an arbitrary $z'\in K$, we define$$
    %c':=\phi_{z'}(x(0)),
    %\mbox{ and }
    \mathcal{A}:=\{\alpha\in K \mid \phi_{z'}(\alpha)=\phi_{z'}(x(0))\}.
    $$
Since $\phi_{z'}$ is strictly convex and $K$ is an open line segment, the restriction of $\phi_{z'}$ to $K$ is strictly convex, which implies that the set $\mathcal A$ contains at most two points.

For the case where $\mathcal{A}=\varnothing$, we can simply let $z''$ be any point in $K-\{ z'\}$, and the proof is done; indeed, if there exists a point
$z$ in $\Gamma_{z',z''}(x(0))\cap K$, then $\phi_{z'}(z)=\phi_{z'}(x(0))$, which implies that $z\in\mathcal A$, contradicting the fact that $\mathcal{A}=\varnothing$. 

We next deal with the case where $\mathcal{A}$ has only one point $\alpha$. 
First, note that $\alpha \neq z'$. This holds because $\phi_{z'}(\alpha) = \phi_{z'}(x(0)) > 0$ while $\phi_{z'}(z') = 0$ by item~1 of Lemma~\ref{lem:phi_z}.
Let $z'':= \alpha$. We show below that $\Gamma_{z',z''}(x(0))$ does not intersect $K$. Suppose that, say $z\in \Gamma_{z',z''}(x(0)) \cap K$; then,
\begin{equation}
\label{eq:phi_two_z}
\phi_{z'}(z) = \phi_{z'}(x(0)) \quad \mbox{and} \quad \phi_{z''}(z) = \phi_{z''}(x(0)).
\end{equation}
On the one hand, the first equality above, together with the fact that {  $\mathcal{A} = \{\alpha\}$}, implies that $z$ can only be $z = \alpha$. On the other hand, the second equality implies that $z\neq z''$ since $ \phi_{z''}(x(0)) > 0$, which is a contradiction. 

Finally, we deal with the case where $\mathcal{A}$ has two points and we write $\mathcal{A} = \{\alpha',\alpha''\}$. We introduce the following sets: 
 $$
 \begin{aligned}
 \mathcal{B}_{\alpha'} & :=\{z\in K \mid \phi_z(\alpha')=\phi_z(x(0))\}, \\
 \mathcal{B}_{\alpha''}& :=\{z\in K \mid \phi_z(\alpha'')=\phi_z(x(0))\}.
 \end{aligned}
 $$
 We first show that each of the above two sets contains at most one point.  
We prove the statement for $\mathcal{B}_{\alpha'}$, and the arguments will be exactly the same for the other.       
Let $z^+$ and $z^-$ be the two ending points of the open segment $K$. Then, for any $z\in K$, there exists a unique $c\in (0,1)$ such that $z=(1- c) z^{-}+cz^{+}$. In other words, we can parameterize the points of $K$ by $z(c):= (1- c) z^{-}+cz^{+}$. 
Consider the following affine function $\psi: (0,1)\to \R$ given by 
    \begin{equation*}
    \psi(c) :=\phi_{z(c)}(x(0))-\phi_{z(c)}(\alpha')
     =-\sum_{i=1}^4((1-c)z^{-}_i+cz^{+}_i)\ln(x_i(0)/\alpha'_i).
    \end{equation*}  
    We claim that the affine function is not identically zero. 
    To wit, first note that $\alpha'\in \mathcal{A}\subseteq K$, so there exists a unique 
   $c_{\alpha'}\in (0,1)$  such that $\alpha'=(1-c_{\alpha'})z^{-}+c_{\alpha'}z^{+}$.
     We evaluate $\psi$ at $c_{\alpha'}$ and obtain that
    $$\psi(c_{\alpha'})=\phi_{\alpha'}(x(0))-\phi_{\alpha'}(\alpha')=\phi_{\alpha'}(x(0))>0,$$
    where the last inequality follows from the fact that $x(0)\neq \alpha'$
    (this holds because $x(0)\notin K$ and $\alpha'\in K$) and item~1 of Lemma~\ref{lem:phi_z}. 
    We have thus established the claim.  
    Consequently, the equation $\psi(c)=0$ has at most one root in $(0,1)$, which implies that the set $\mathcal{B}_{\alpha'}$
    has at most one point. 

    We now let $z''$ be any point in $ K-(\mathcal{B}_{\alpha'}\cup\mathcal{B}_{\alpha''}\cup\{z'\})$, and prove that $\Gamma_{z',z''}(x(0))\cap K=\varnothing$. {  Suppose, to the contrary, that $$ \Gamma_{z',z''}(x(0))\cap K\neq\varnothing; $$ then, there exists \(z\in \Gamma_{z',z''}(x(0))\cap K\), and hence \(z\) satisfies~\eqref{eq:phi_two_z}. }%Suppose not, say $z\in \Gamma_{z',z''}(x(0))\cap K$; then, $z$ satisfies~\eqref{eq:phi_two_z}.
    The first equality of~\eqref{eq:phi_two_z} implies that $z\in \mathcal{A}$ (i.e., either $z = \alpha'$ or $z = \alpha''$) 
    while the second equality implies that $z''\in \mathcal{B}_{z}$, which contradicts the fact that $z''\notin \mathcal{B}_{\alpha'}\cup\mathcal{B}_{\alpha''}$.  
    This completes the proof. 
\end{proof}

In the sequel, we will fix $x(0)\in \rint \Delta^3 - K$ and the two distinct points $z', z''\in K$ which satisfy Lemma~\ref{lem:two_z}. 
For ease of notation, we will simply use $\Gamma$ to denote $\Gamma_{z',z''}(x(0))$. 
%We now show that for any such choice of $z'$ and $z''$ in the above lemma and any given $x(0)\in \rint\Delta^3-K$, 
We now have the following result: %show that $\Gamma$ is a smooth one-dimensional submanifold.

\begin{Lemma}\label{lem:regular_value}
%Let $x(0)\in\rint\Delta^3-K$, and two distinct points $z',z''\in K$ such that $\Gamma \cap K=\varnothing.$
% Then $dF_{x}:T_{x}\Delta^3\to\mathbb R^2$ has rank $2$ for all $x\in \Gamma$,
The set $\Gamma$ is a compact, smooth one-dimensional submanifold of
$\rint\Delta^3$ without boundary. In particular, $\Gamma$ has finitely many connected components, each of which is diffeomorphic to~$S^1$. 
\end{Lemma}

\begin{proof}
% For each $x\in\rint\Delta^3$, the tangent space of $\Delta^3$ at $x$ is
% $$
% T_x\Delta^3=\{v\in\mathbb{R}^4:\mathbf{1}^\top v=0\}.
% $$
We first show that $\Gamma$ is a compact subset of $\rint \Delta^{3}$ and then, 
show that $\Gamma$ is {  a smooth one-dimensional submanifold} of
$\rint\Delta^3$ without boundary. 

{\it{Proof that $\Gamma$ is compact.}}
We first show that for any $z\in K$ and any $c > 0$, $\phi_z^{-1}(c)$ is a compact subset of $\rint \Delta^3$. 
By item~1 of Lemma~\ref{lem:phi_z}, the function $\phi_{z}$ is strictly convex on $\rint\Delta^3$ and $\phi_{z}(x)$ goes to infinity as $x$ approaches $\partial\Delta^3$.  
It then follows 
that the sublevel set $\phi_{z}^{-1}([0,c])$ is convex and, moreover, contained in closed subset of $\rint \Delta^3$ and hence, is itself closed. 
It follows that $\phi_z^{-1}(c)$, being the boundary of $\phi_{z}^{-1}([0,c])$, is homeomorphic to $S^2$~\cite[Theorem 16.4]{bredon2013topology}.  
{ Thus, $\Gamma =\phi^{-1}_{z'}(\phi_{z'}(x(0))) \cap \phi^{-1}_{z''}(\phi_{z''}(x(0)))$} is compact. 
% closed and convex. 
% Since $\phi_{z}(x)$ goes to infinity as $x$ approaches $\partial\Delta^3$, the set $\phi_{z}^{-1}([0,c])$ is bounded away from $\partial\Delta^3$, so $\phi_{z}^{-1}([0,c])$ is contained in a compact subset of $\rint\Delta^3$. By the continuity of $\phi_{z}$, $\phi_{z}^{-1}([0,c])$ is closed in $\rint \Delta^3$, and thus compact. This also implies that the level set $\phi^{-1}_{z}(c)$ is a closed subset of $\phi_{z}^{-1}([0,c])$, and is compact. Consequently, $\Gamma$, being the intersection of two compact sets, and its connected component $\gamma$ are both compact.

{\it{Proof that $\Gamma$ is a smooth one-dimensional submanifold.}}
By the regular value theorem~\cite[Chapter 1]{guillemin2010differential}, it suffices to show that the Jacobian $d_xF:T_{x}\Delta^3\to\mathbb \R^2$ has full rank (rank $2$) for any $x \in \Gamma$. The proof is carried out by contradiction. 
 Suppose, to the contrary, that there exists an $x\in \Gamma$ such that $d_x\phi_{z'}$ and $d_x\phi_{z''}$
are linearly dependent in $T_{x}\Delta^3$, which is the subspace of $\R^4$ perpendicular to~$\mathbf{1}$. Then, there exist scalars $\lambda, \mu \in \R$ such that
$$
\frac{\partial \phi_{z'}(x)}{\partial x} - \lambda\frac{\partial \phi_{z''}(x)}{\partial x} = \mu \mathbf{1}.
$$
By computation, the above equation can hold if and only if
% $$
% -\frac{z_i'}{x_i} + \lambda\frac{z_i''}{x_i} = \mu, \mbox{ for any }i\in\{1,2,3,4\},$$ 
% which implies that
$$
\lambda z_i'' - z_i' = \mu x_i, \mbox{ for any }i\in\{1,2,3,4\}.
$$
If $\mu=0$, then $\lambda z''=z'$, which implies $z'= z''$, contradicting the fact that $z' \neq z''$. 
We thus assume that $\mu\neq 0$. But then, 
$x = \frac{\lambda}{\mu} z'' - \frac{1}{\mu} z'$, which is a linear combination of $z'$ and $z''$. It follows that $x$ belongs to the null space of $A$. Further, since $x\in \Gamma$ and $\Gamma \subseteq \rint \Delta^3$, we have that $x\in \nll(A) \cap \rint \Delta^3 = K$. 
However, this contradicts the fact that $x\in \Gamma$ and the fact that $\Gamma\cap K=\varnothing$, the latter of which follows from Lemma~\ref{lem:two_z}. 
Therefore, $d_xF$ has rank $2$ for any $x\in\Gamma$, and $F_{z',z''}(x)$ is a regular value of $F_{z',z''}$. Since
$\rint\Delta^3$ is a smooth three-dimensional manifold without boundary,
the regular value theorem
 implies that $\Gamma$ 
is a smooth one-dimensional submanifold of $\rint\Delta^3$
without boundary.

We conclude that every connected, compact one-dimensional manifold without
boundary is diffeomorphic to $S^1$. Since $\Gamma$ is compact, it has finitely
many such connected components. 
\end{proof}
With the lemmas above, we are now ready to prove Theorem~\ref{thm:periodic_orbit}:

% {\color{red} change the format}
\begin{proof}[Proof of Theorem~\ref{thm:periodic_orbit}.]
We first show that for any $x(0)\in \rint \Delta^3-K$, the trajectory through $x(0)$ is a periodic orbit.
%\xc{Proof that each trajectory is a periodic orbit.} 
We recall that $z'$ and $z''$ are two distinct points of $K$ such that 
$\Gamma = \Gamma_{z',z''}(x(0))$ does not intersect $K$. 
By Lemma~\ref{lem:phi_z}, both $\phi_{z'}$ and $\phi_{z''}$ are invariant 
along dynamics~\eqref{eq:zsreplicator}. 
Since the trajectory is continuous, $x(t)$ remains in the connected component
$\gamma$ for all $t\geq 0$. By Lemma~\ref{lem:regular_value}, $\gamma$ is diffeomorphic to $S^1$. The above arguments imply that $\gamma$ is a periodic orbit.  

% Moreover, the vector field of~\eqref{eq:zsreplicator} is tangent to $\gamma$ and is nowhere vanishing on $\gamma$ since $K$ comprises all the equilibria
% of~\eqref{eq:zsreplicator} in $\rint\Delta^3$ and $\gamma\cap K=\varnothing$. 
% Therefore, the trajectory through $x(0)$ traverses the entire component $\gamma$, and is a periodic orbit.

% By Lemma~\ref{lem:regular_value}, $\Gamma$ is a compact 1-manifold without boundary. By classification of 1-manifolds, $\Gamma$ consists of a finite number of disjoint connected components, each diffeomorphic to $S^1$. Since the solution $x(t)$ is continuous with respect to time, it cannot jump between disjoint components. Therefore, $x(t)$ remains entirely within the connected component containing $x(0)$. Since $\Gamma \cap K = \varnothing$, the vector field is non-vanishing on this component, implying the induced flow is periodic.
% \end{proof}

% \begin{Theorem}
%     Let $x(0)\in \rint \Delta^3- K$, and $\gamma$ be the connected component of $\Gamma$ containing $x(0)$. Then $\gamma$ is orbitally stable.
% \end{Theorem}
% \xc{Proof that $\gamma$ is stable.}
We now show that $\gamma$ is stable. 
Specifically, given any $\varepsilon > 0$, we show that there exists an open neighborhood $U$ of $\gamma$ in $\rint \Delta^3$ such that if $x(0)\in U$, then $\dist(x(t),\gamma) < \varepsilon$.  
For any $\delta > 0$, we let 
$$
N_\delta:=\{x\in\rint\Delta^3\mid\dist(x,\gamma)\leq \delta\}.
$$
We let $\delta$ be sufficiently small so that ({\it i}) $\delta\leq \epsilon$, ({\it ii})
$N_\delta$ is a closed tubular neighborhood of $\gamma$ in $\rint\Delta^3$, and ({\it iii}) $N_\delta$ does not intersect any other connected component of $\Gamma$ (this is feasible since $\Gamma$ has only finitely many connected components, all of which are diffeomorphic to $S^1$). 

For convenience, let $c':=\phi_{z'}(x(0))$ and $c'':=\phi_{z''}(x(0))$. Consider the Lyapunov function $V: \rint \Delta^3\to \R_{\geq 0}$ as 
$$V(x):=(\phi_{z'}(x)-c')^2+(\phi_{z''}(x)-c'')^2.$$ 
It is clear that $V(x)=0$ if and only if $x\in \Gamma$ and that $V(x)\to \infty$ as $x\to \partial \Delta^3$. 

Since $\partial N_\delta := \{x\in N_\delta \mid \dist(x,\gamma) = \delta\}$ is compact and since it does not intersect $\Gamma$, we have that 
\begin{equation}
\label{eq:def_rho}
    \rho:= \min_{x\in\partial N_{\delta}}V(x) > 0.
\end{equation} 
We now let
$$
U:= V^{-1}(-\infty,\rho) \cap \rint N_\delta, 
$$
which is an open set. It now remains to show that $U$ is forward invariant. First, note that the function $V$ is invariant; indeed, 
\begin{equation*}
\dot V(x(t)) = 2(\phi_{z'}(x(t)) - c')\dot \phi_{z'}(x(t)) + 2(\phi_{z''}(x(t)) - c'')\dot \phi_{z''}(x(t)) = 0,
\end{equation*}
where the last equality follows from item~2 of Lemma~\ref{lem:phi_z}. We next show that if $x(0)\in U$, then $x(t)\in \rint N_\delta$ for all $t$. Suppose not, say $x(t^*) \in \partial N_\delta$ for some $t^* > 0$; then, by~\eqref{eq:def_rho}, $V(x(t^*)) \geq\rho > V(x(0))$, which contradicts the fact that $V(x(t^*))  = V(x(0))$. 
\end{proof}

\section{Dynamics on the Boundary and Stability Property}
\label{sec:sys_boundary}

In this section, we provide a complete characterization of the boundary dynamics of replicator dynamics induced by $4$-by-$4$ payoff matrices when coexistence occurs. 
Since each face of $\Delta^3$ is a two-dimensional simplex and each edge is forward invariant under~\eqref{eq:zsreplicator}, the boundary dynamics can be analyzed by combining the sign patterns of the skew-symmetric payoff matrix with the known phase-portrait classification for replicator dynamics with two or three strategies. By Theorem~\ref{thm:main2}, coexistence occurs if and only if the associated digraph $\mathcal{G}_A$ belongs to one of the five isomorphism classes shown in Fig.~\ref{fig:combined_cases}. We examine these five cases one by one and characterize, for each class, the equilibria on the boundary together with the asymptotic behavior on the forward invariant edges and faces.

{  
To proceed, we first note that for each~$i \in \{1,2,3,4\}$, the face $F_{-i}$ is forward invariant. Thus, the dynamics~\eqref{eq:zsreplicator}, when restricted to $F_{-i}$, is reduced to the three-strategy replicator dynamics induced by the principal submatrix $A_{-i}$ of $A$, where $A_{-i}$ is obtained from $A$ by removing the $i$-th row and the $i$-th column. 
Correspondingly, $\mathcal{G}_{A_{-i}}$ is the subgraph of $\mathcal{G}_A$ induced by $\{1,2,3,4\} - \{i\}$.

\begin{figure*}[t]
\centering
\begin{tabular}{cccccc}

%================================================
% (f) Directed 3-cycle
%================================================
\begin{subfigure}[t]{0.15\textwidth}
\centering
\begin{tikzpicture}[
  baseline=(current bounding box.center),
  scale=1.7,
  v/.style={circle,fill=black,inner sep=1.3pt},
  e/.style={line width=0.8pt,->,>=Stealth}
]
  \node[v,label=below:$j$] (n1) at (0,0) {};
  \node[v,label=below:$k$] (n2) at (1,0) {};
  \node[v,label=above:$\ell$] (n3) at (0.5,{sqrt(3)/2}) {};
  \draw[e] (n1) -- (n2);
  \draw[e] (n2) -- (n3);
  \draw[e] (n3) -- (n1);
\end{tikzpicture}
\caption*{(a)}
\phantomsubcaption
\makeatletter
\def\@currentlabel{a}
\makeatother
\label{subfig:boundary-i}
\end{subfigure}
&
%================================================
% (e) Transitive three-edge graph
%================================================
\begin{subfigure}[t]{0.15\textwidth}
\centering
\begin{tikzpicture}[
  baseline=(current bounding box.center),
  scale=1.7,
  v/.style={circle,fill=black,inner sep=1.3pt},
  e/.style={line width=0.8pt,->,>=Stealth}
]
  \node[v,label=below:$j$] (n1) at (0,0) {};
  \node[v,label=below:$k$] (n2) at (1,0) {};
  \node[v,label=above:$\ell$] (n3) at (0.5,{sqrt(3)/2}) {};
  \draw[e] (n2) -- (n1);
  \draw[e] (n3) -- (n1);
  \draw[e] (n2) -- (n3);
\end{tikzpicture}
\caption*{(b)}
\phantomsubcaption
\makeatletter
\def\@currentlabel{b}
\makeatother
\label{subfig:boundary-ii}
\end{subfigure}
&
%================================================
% (d) Two edges with a common sink
%================================================
\begin{subfigure}[t]{0.15\textwidth}
\centering
\begin{tikzpicture}[
  baseline=(current bounding box.center),
  scale=1.7,
  v/.style={circle,fill=black,inner sep=1.3pt},
  e/.style={line width=0.8pt,->,>=Stealth}
]
  \node[v,label=below:$j$] (n1) at (0,0) {};
  \node[v,label=below:$k$] (n2) at (1,0) {};
  \node[v,label=above:$\ell$] (n3) at (0.5,{sqrt(3)/2}) {};
  \draw[e] (n2) -- (n1);
  \draw[e] (n3) -- (n1);
\end{tikzpicture}
\makeatletter
\def\@currentlabel{c}
\makeatother
\caption*{(c)}
\label{subfig:boundary-iii}
\end{subfigure}
&
%================================================
% (c) Two edges forming a directed path
%================================================
\begin{subfigure}[t]{0.15\textwidth}
\centering
\begin{tikzpicture}[
  baseline=(current bounding box.center),
  scale=1.7,
  v/.style={circle,fill=black,inner sep=1.3pt},
  e/.style={line width=0.8pt,->,>=Stealth}
]
  \node[v,label=below:$j$] (n1) at (0,0) {};
  \node[v,label=below:$k$] (n2) at (1,0) {};
  \node[v,label=above:$\ell$] (n3) at (0.5,{sqrt(3)/2}) {};
  \draw[e] (n2) -- (n1);
  \draw[e] (n1) -- (n3);
\end{tikzpicture}
\caption*{(d)}
\makeatletter
\def\@currentlabel{d}
\makeatother
\label{subfig:boundary-iv}
\end{subfigure}
&
%================================================
% (b) Two edges with a common source
%================================================
\begin{subfigure}[t]{0.15\textwidth}
\centering
\begin{tikzpicture}[
  baseline=(current bounding box.center),
  scale=1.7,
  v/.style={circle,fill=black,inner sep=1.3pt},
  e/.style={line width=0.8pt,->,>=Stealth}
]
  \node[v,label=below:$j$] (n1) at (0,0) {};
  \node[v,label=below:$k$] (n2) at (1,0) {};
  \node[v,label=above:$\ell$] (n3) at (0.5,{sqrt(3)/2}) {};
  \draw[e] (n1) -- (n2);
  \draw[e] (n1) -- (n3);
\end{tikzpicture}
\caption*{(e)}
\makeatletter
\def\@currentlabel{e}
\makeatother
% \label{subfig:boundary-v}
\label{subfig:boundary-v}
\end{subfigure}
&
%================================================
% (a) One directed edge
%================================================
\begin{subfigure}[t]{0.15\textwidth}
\centering
\begin{tikzpicture}[
  baseline=(current bounding box.center),
  scale=1.7,
  v/.style={circle,fill=black,inner sep=1.3pt},
  e/.style={line width=0.8pt,->,>=Stealth}
]
  \node[v,label=below:$j$] (n1) at (0,0) {};
  \node[v,label=below:$k$] (n2) at (1,0) {};
  \node[v,label=above:$\ell$] (n3) at (0.5,{sqrt(3)/2}) {};
  \draw[e] (n1) -- (n2);
\end{tikzpicture}
\caption*{(f)}
\makeatletter
\def\@currentlabel{f}
\makeatother
% \label{subfig:boundary-ii}
\label{subfig:boundary-vi}
\end{subfigure}
\end{tabular}

\caption{The six isomorphism classes of the induced subgraphs $\mathcal{G}_{A_{-i}}$.}
\label{fig:boundary_graphs}
\end{figure*}

By Theorem~\ref{thm:fineclass}, we have six classes of possible digraphs for $\mathcal{G}_{A_{-i}}$ as shown in Fig.~\ref{fig:boundary_graphs}.  
The classification of all three-strategy replicator dynamics with constant $3$-by-$3$ payoff matrices has been completed by Bomze~\cite{Bomze1983,Bomze1995}. Relying upon his results, we have the following lemma: 
%The associated interaction graph $\mathcal{G}$ is the subgraph of $\mathcal G_A$ induced by the three strategies present on $F_{-i}$, with its directed edges representing the nonzero pairwise interactions among these strategies.

% {  1. draw all the five figures of $\mathcal{G}_{A_{-i}}$

% 2. give a unified lemma for all the behaviors on the boundary 

% 3. in the second section, recall each item and prove stability

% }

\begin{Lemma}
\label{lem:face_lem}
    %Let $\mathcal{G}_{A_{-i}}$ be the associated subgraph of $\mathcal{G}_A$ for face $i$. The dynamic behavior of~\eqref{eq:zsreplicator} restricted to face $F_{-i}$ belongs to one of the following cases. After relabeling, we have
    The following items hold:
    \begin{enumerate}
        \item If $\mathcal{G}_{A_{-i}}$ is the digraph~(\ref{subfig:boundary-i}) in Fig.~\ref{fig:boundary_graphs}, then any $x(0)\in \rint F_{-i}-L$ belongs to a stable periodic orbit.
        
        \item If $\mathcal{G}_{A_{-i}}$ is the digraph~(\ref{subfig:boundary-ii}) in Fig.~\ref{fig:boundary_graphs}, then for any $x(0)\in \rint F_{-i}$, $\lim_{t\to \infty} {x(t)} = e_{j}$.
        
        \item If $\mathcal{G}_{A_{-i}}$ is the digraph~(\ref{subfig:boundary-iii}) in Fig.~\ref{fig:boundary_graphs}, then for any $x(0)\in \rint F_{-i}$, $\lim_{t\to \infty} {x(t)} = e_{j}$.
        
        \item If $\mathcal{G}_{A_{-i}}$ is the digraph~(\ref{subfig:boundary-iv}) in Fig.~\ref{fig:boundary_graphs}, then for any $x(0)\in \rint F_{-i}$,
        there exists a unique $z\in \rint E_{k\ell}$ satisfying $a_{jk} z_k+a_{j\ell} z_\ell<0$, and 
        $$a_{j\ell}\log z_k-a_{jk}\log z_\ell = a_{j\ell}\log x_k(0)-a_{jk}\log x_\ell(0), $$ %with $z_1\in (a_{12}/(a_{12}-a_{13}),1)$, 
        such that $\lim_{t\to\infty} x(t)=z$.
        
        \item If $\mathcal{G}_{A_{-i}}$ is the digraph~(\ref{subfig:boundary-v}) in Fig.~\ref{fig:boundary_graphs}, then for any $x(0) \in \rint F_{-i}$, 
        there exists a unique $z\in \rint E_{k\ell}$, satisfying 
        $$a_{j\ell}\log z_k-a_{jk}\log z_\ell = a_{j\ell}\log x_k(0)-a_{jk}\log x_\ell(0),$$ such that $\lim_{t\to\infty} x(t)=z$.

        \item If $\mathcal{G}_{A_{-i}}$ is the digraph~(\ref{subfig:boundary-vi}) in Fig.~\ref{fig:boundary_graphs}, then for any $x(0) \in \rint F_{-i}$, there exists a unique $z\in \rint E_{k\ell}$, with $z_\ell=x_\ell(0)$, such that $
        \lim_{t\to\infty} x(t)=z$.
    \end{enumerate}
\end{Lemma}

Up to relabeling of the strategies and reversal of the sign for the payoff, the six cases above correspond, respectively, to cases~16, 43, 29, 33, 29, and~18 in~\cite[Fig.~6]{Bomze1983}. However, the author of the paper did not provide detailed proof for each case. 
For completeness of the presentation, we provide a proof of the lemma in the Appendix.

Also, note that each edge $E_{ij}$ of the simplex is forward-invariant, the dynamics~\eqref{eq:zsreplicator}, when restricted to $E_{ij}$ will be reduced to a two-strategy replicator dynamics. 
If $a_{ij} = 0$, then each point of $E_{ij}$ is an equilibrium point.
If $a_{ij} \neq 0$, then for any $x(0)\in \rint E_{ij}$, we have that 
\begin{equation}
\label{eq:limit_edge}
\lim_{t\to\infty} x(t)=
    \begin{cases}
    e_i, & a_{ij}>0,\\
    e_j, & a_{ij}<0.
    \end{cases}
\end{equation}
With the preliminaries above, we will now describe the boundary dynamics and the associated stability properties for the five cases in Fig.~\ref{fig:combined_cases} in five subsequent propositions. 

More specifically, for each case, we will describe in the statement of the corresponding proposition {\it (1)} 
the $\omega$-limiting set of the dynamics~\eqref{eq:zsreplicator} in $\partial \Delta^3$ (which comprises equilibria and possible periodic orbits), {\it (2)} the stability properties of these equilibria and periodic orbits, and {\it (3)} the dynamic behavior of any trajectory whose initial condition $x(0)$ belongs to the interior of a certain face. The exhibition of the equilibria follows straightforwardly from the computation $\diag(x) A x = 0$. The dynamic behavior follows
from Lemma~\ref{lem:face_lem}. We will thus only provide proofs about stability properties of the equilibria and periodic orbits (if they exist). Further, note that the stability properties of the two equilibria in the set $L\cap \partial \Delta^3$ have already been addressed in Lemma~\ref{lem:L_stability}---each $x\in L\cap \partial \Delta^3$ is stable, but not asymptotically stable.   
}
% {\color{red}
% Note that the set of equilibria can be verified via checking the dynamics in~\eqref{eq:zsreplicator} as $\dot x_i=0$ for all $i=1,2,3,4$, for which we omit the proof here (and also for the following propositions).
% }

% \begin{proof}
%     Consider the function as defined in~\eqref{eq:def_phi} with $z\in L\cap \Delta^3$. We adopt the convention that $\ln x_i/z_i=0$ whenever $z_i=0$. Then for any $x\in L\cap \Delta^3$, by the log-sum inequality, we have 
%     $$\phi_z(x)= -\sum_{i=1}^4 z_i\log\frac{x_i}{z_i}
%     \geq -\log(\sum_{i=1}^4 x_i) \geq0,$$
%     with equality holds if and only if $x=z$. Moreover, Lemma~\ref{lem:phi_z} has shown that $\dot \phi_z(x)=0$. Apply the Lyapunov Theorem, we have that any $x\in L\cap \Delta^3$ is stable. 

%     It remains to show that such $x$ is not asymptotically stable. Since $L\cap\Delta^3$ is a nontrivial line segment consisting of the continuum of equilibria, every relative neighborhood of $z$ contains another point $z'\in L\cap\Delta^3$ with $z'\neq z$, such that $x(t)=z'$ for any $t\geq 0$. Hence every point $L\cap \Delta^3$ in is not asymptotically stable. This completes our proof.
% \end{proof}
% 
% {\it{~\ref{subfig:case1}:} $z^+\in \rint(F_{-i})$, $z^-\in \rint(F_{-j})$.} Fix the index $i<j, k<\ell$, we have 

% Without loss of generality, we set $i=1,j=2$. 

% \subsection{Phase Portraits of Case~(\ref{subfig:I})}
    %Given a skew-symmetric $A$ with $\det(A)=0$, and 

   {  
    \begin{proposition}
        Let $\mathcal{G}_A$ be the digraph in~(\ref{subfig:I}) of Fig.~\ref{fig:combined_cases}. 
        Then, the following items hold:
        \begin{enumerate}
            \item In this case, $L$ intersects $\partial \Delta^3$ inside $F_{-1}$ and $F_{-2}$. 
            The set of equilibria of dynamics~\eqref{eq:zsreplicator} in $\partial \Delta^3$ is given by 
            $$\{e_1,e_2,e_3,e_4\}\cup (L\cap \partial \Delta^3).$$ 
            \item Each vertex $e_i$ is unstable. Each $x\in L\cap \partial\Delta^3$ is stable but not asymptotically stable. 
            \item The following hold on the faces of $\Delta^3$:
        \begin{enumerate}
        \item If $i\in \{1,2\}$, then any $x \in \rint F_{-i}-L$ belongs to a stable periodic orbit.  
        \item
        For any $x(0) \in \rint F_{-3}$, $\lim_{t\to\infty} x(t)=e_4$. %see, e.g.,~\cite[case~2 in Fig.11]{Zeeman1980} and~\cite[case 43 in Figure 6]{Bomze1983}.

        \item 
        For any $x(0) \in \rint F_{-4}$, $\lim_{t\to\infty} x(t)=e_1$. %see, e.g.,~\cite[case~2 in Fig.11]{Zeeman1980} and~\cite[case 43 in Figure 6]{Bomze1983}.
        \end{enumerate}
        \end{enumerate} 
    \end{proposition}

    \begin{proof} 
    We show that any $x\in \{e_1,e_2,e_3,e_4\}$ is unstable. For each $i = 1,2,3,4$, define 
    $$\sigma(i):=
    \begin{cases}
     4 & \mbox{if } i = 1; \\
     1 & \mbox{if } i = 2; \\
     2 & \mbox{if } i = 3; \\
     3 & \mbox{if } i = 4. 
    \end{cases}
    $$ 
    It then follows from the sign pattern of $A$ (represented by Graph~(\ref{subfig:I}) in Fig.~\ref{fig:combined_cases}) and~\eqref{eq:limit_edge} that if $x(0)\in \rint E_{i\sigma(i)}$, then $\lim_{t\to\infty} x(t) = e_{\sigma(i)} \neq e_i$. 
    \end{proof}
\begin{proposition}
Let $\mathcal{G}_A$ be the digraph in~(\ref{subfig:II}) of Fig.~\ref{fig:combined_cases}. Then, the following items hold:

\begin{enumerate}
    \item 
    In this case, $L$ intersects $\partial \Delta^3$ inside $F_{-1}$ and $F_{-2}$. 
    The set of equilibria of dynamics~\eqref{eq:zsreplicator} in $\partial\Delta^3$ is given by 
    $$E_{12}\cup \{e_3,e_4\}\cup (L\cap \partial \Delta^3).$$ \item Each $x\in \{e_3,e_4\}\cup E_{12}$ is unstable. Each $x\in L\cap \partial\Delta^3$ is stable but not asymptotically stable.
    \item The following hold on the faces of $\Delta^3$:
    \begin{enumerate}
        \item If $i\in \{1,2\}$, then any $x \in \rint F_{-i}-L$ belongs to a stable periodic orbit.
        
        \item
        For any $x(0) \in \rint F_{-3}$, $
        \lim_{t\to\infty} x(t)=e_4
        $.
        
        \item
        For any $x(0) \in \rint F_{-4}$, 
        there exists a unique $z\in \rint E_{12}$, such that $\lim_{t\to\infty} x(t)=z$.
    \end{enumerate}
\end{enumerate}
\end{proposition}

\begin{proof}
    We show that any $x\in \{e_3,e_4\}\cup E_{12}$ is unstable. We first prove instability of each vertex~$e_i$. For each $i=1,2,3,4$, define 
    $$\sigma(i):=
    \begin{cases}
     4 & \mbox{if } i = 1; \\
     4 & \mbox{if } i = 2; \\
     2 & \mbox{if } i = 3; \\
     3 & \mbox{if } i = 4. 
    \end{cases}
    $$ 
    It then follows from the sign pattern of $A$ (represented by Graph~(\ref{subfig:II}) in Fig.~\ref{fig:combined_cases}) and~\eqref{eq:limit_edge} that if $x(0)\in \rint E_{i\sigma(i)}$, then $\lim_{t\rightarrow\infty}x(t)=e_{\sigma(i)}\neq e_i$. 
    We now prove instability of~$x\in \rint E_{12}$. By item~(3) of Lemma~\ref{lem:face_lem}, 
    we have that for any $x(0)\in \rint F_{-3}$, $\lim_{t\rightarrow \infty}x(t)=e_4\notin \rint E_{12}$.
\end{proof}

\begin{figure*}[!t]
\centering
% ---------- Row 1 ----------
\begin{subfigure}[t]{0.33\textwidth}
  \centering
  \includegraphics[width=\linewidth]{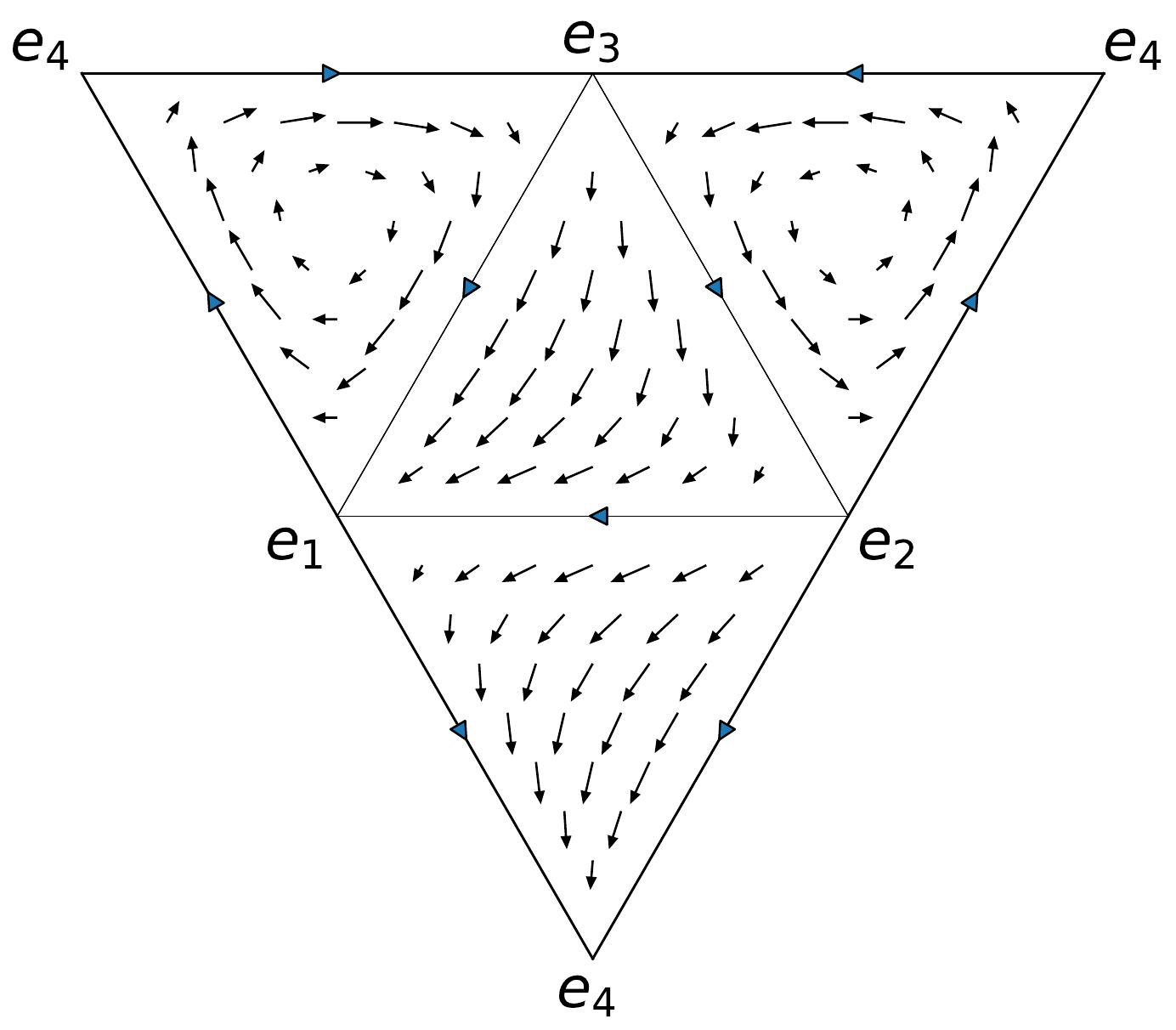}
\caption{~Phase portraits for case~(\ref{subfig:I})}
  % \label{subfig:case_a}
\end{subfigure}\hfill
\begin{subfigure}[t]{0.33\textwidth}
  \centering
  \includegraphics[width=\linewidth]{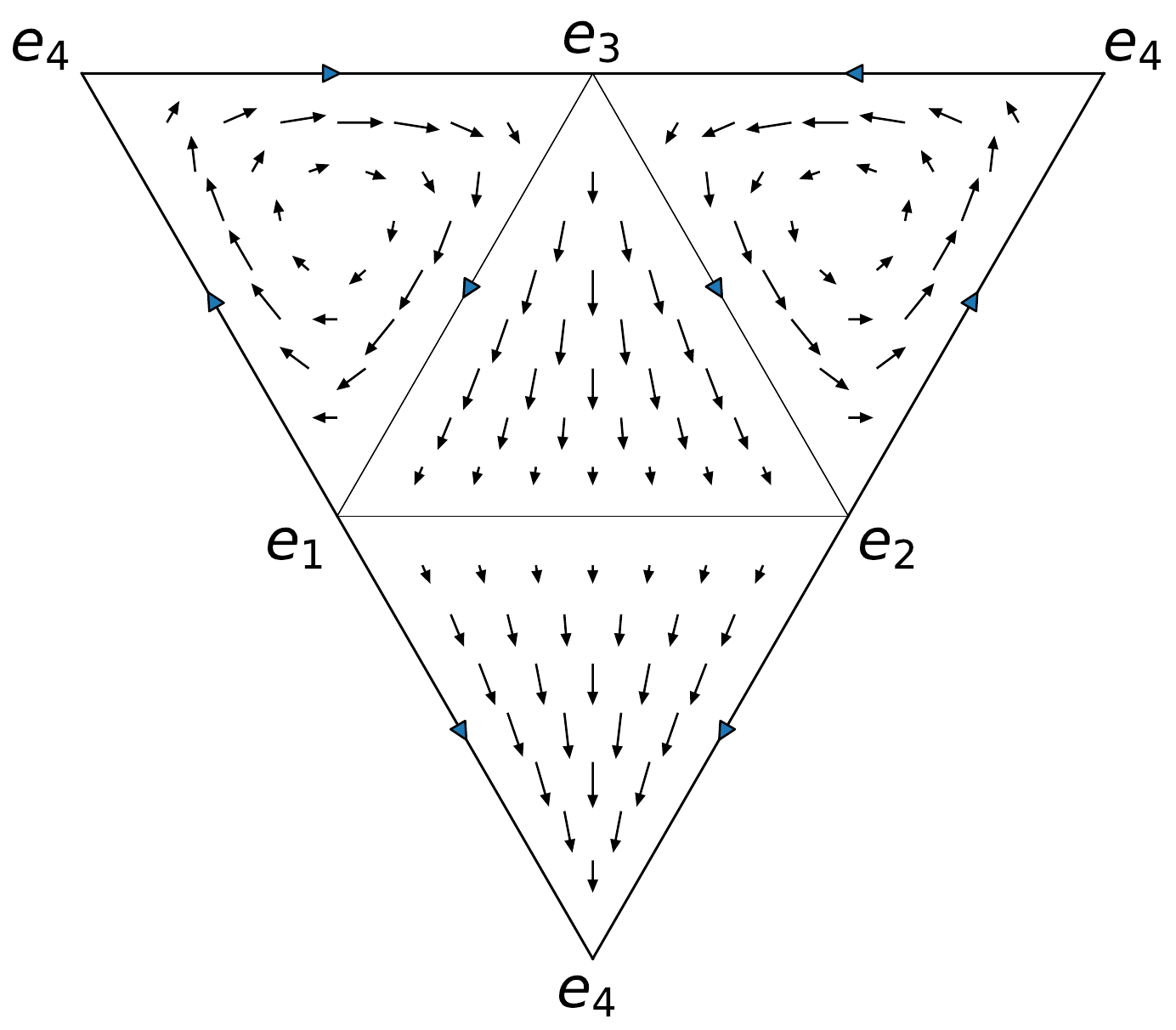}
  \caption{~Phase portraits for case~(\ref{subfig:II})}
  % \label{subfig:case_b}
\end{subfigure}
\hfill
% ---------- Row 2 ----------
\begin{subfigure}[t]{0.33\textwidth}
  \centering
  \includegraphics[width=\linewidth]{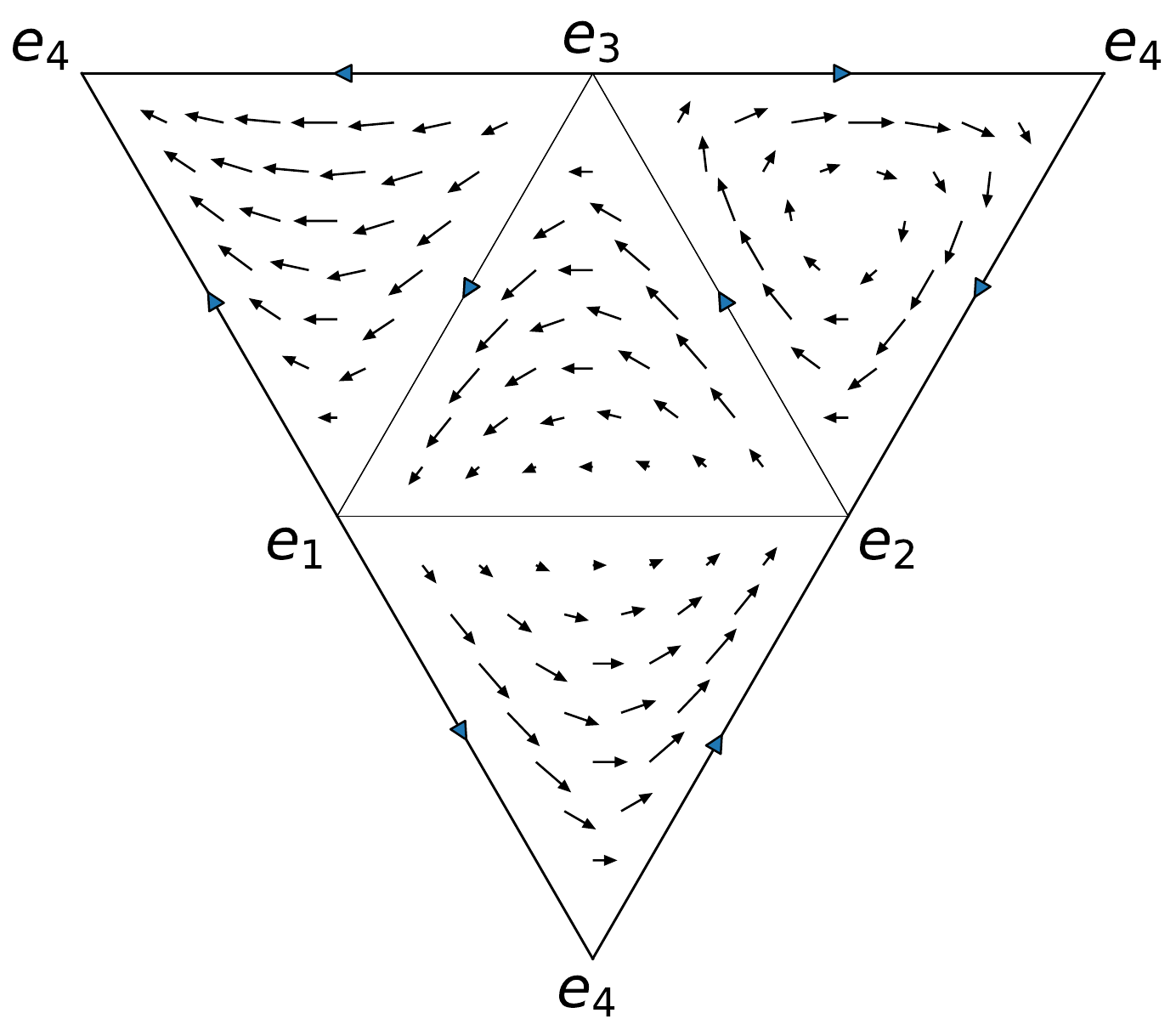}
  \caption{~Phase portraits for case~(\ref{subfig:III})}
  % \label{subfig:case_c}
\end{subfigure}

\vspace{6pt}

% ---------- Row 3 ----------
\begin{subfigure}[t]{0.33\textwidth}
  \centering
  \includegraphics[width=\linewidth]{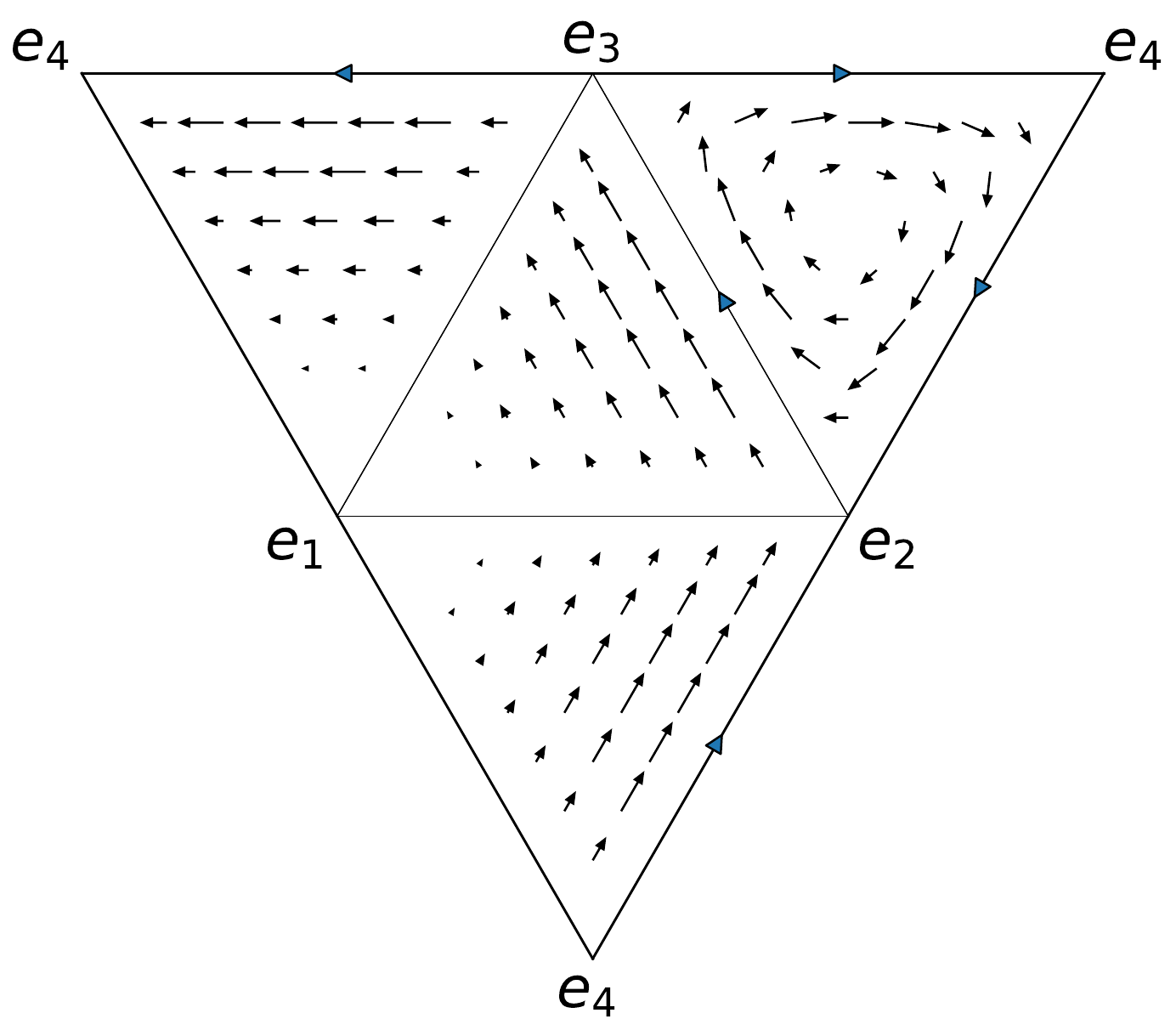}
  \caption{~Phase portraits for case~(\ref{subfig:IV})}
  % \label{subfig:case_d}
\end{subfigure}\quad 
\begin{subfigure}[t]{0.33\textwidth}
  \centering
  \includegraphics[width=\linewidth]{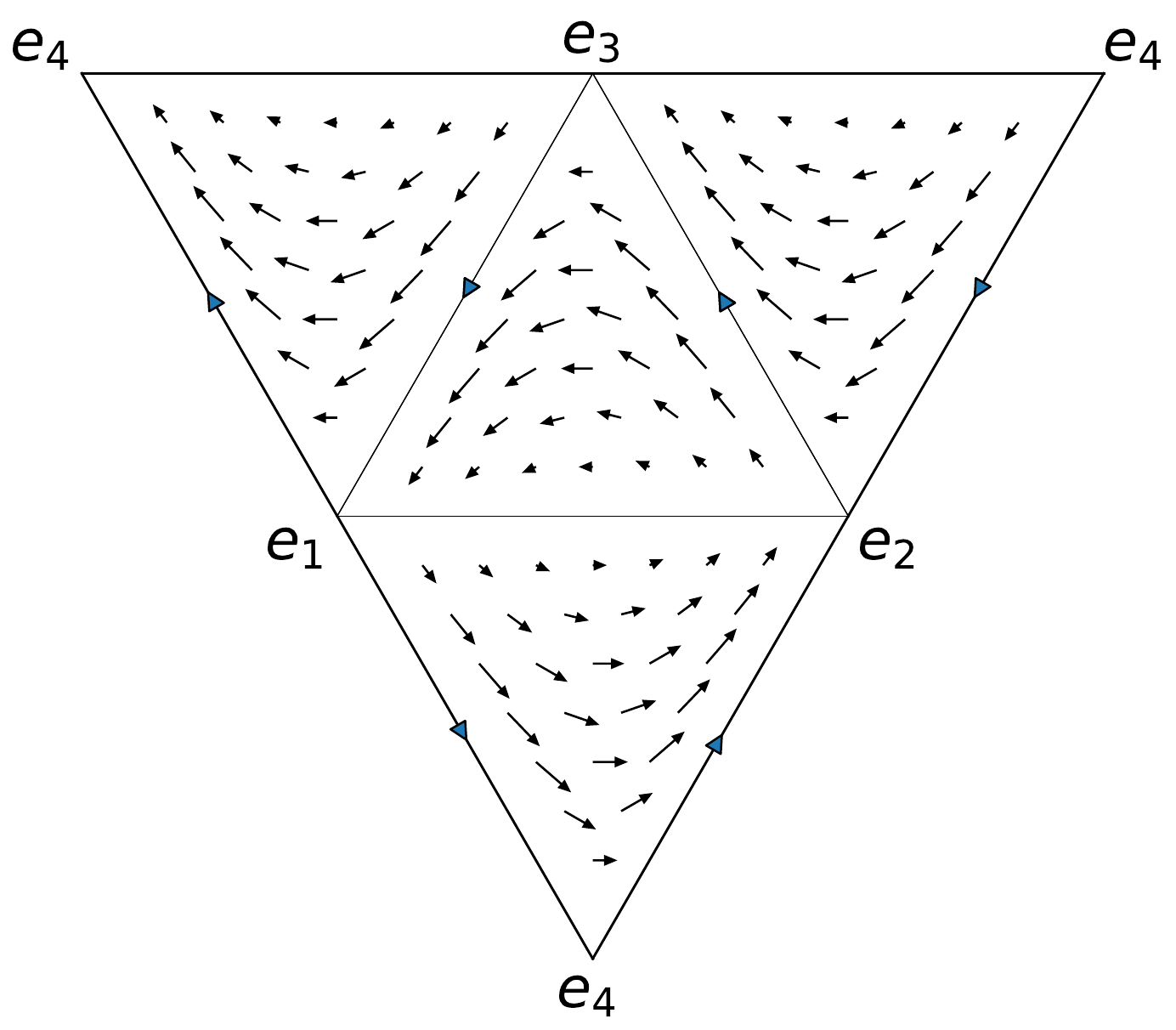}
  \caption{~Phase portraits for case~(\ref{subfig:V})}
  % \label{subfig:case_e}
\end{subfigure}
\caption{Phase portraits for all five classes.}
\label{fig:phase_portraits}
\end{figure*}

% \subsection{Phase Portraits of Case~(\ref{subfig:III})}
% \label{prop:case_13}

\begin{proposition}
\label{prop:iii}
Let $\mathcal{G}_A$ be the digraph in~(\ref{subfig:III}) of Fig.~\ref{fig:combined_cases}. The following items hold:

\begin{enumerate}
    \item 
    In this case, $L$ intersects $\partial \Delta^3$ inside $F_{-1}$ and $E_{12}$. 
    The set of equilibria of dynamics~\eqref{eq:zsreplicator} in $\partial\Delta^3$ is given by $$E_{12}\cup \{e_3,e_4\}\cup (L\cap \partial \Delta^3).$$  
    \item Each $x\in \{e_3,e_4\}\cup (E_{12}-L)$ is unstable. Each $x\in L\cap \partial \Delta^3$ is stable but not asymptotically stable.
    \item The following hold on the faces of $\Delta^3$:
    \begin{enumerate}
        \item Any $x \in \rint F_{-1}-L$ belongs to a stable periodic orbit.
        \item For any $x(0) \in \rint F_{-2}$, $\lim_{t\rightarrow \infty}x(t)=e_4$.
        \item  For any $x(0) \in \rint F_{-3}$, there exists a unique $z\in \rint E_{12}$ with $a_{41} z_1+a_{42} z_2<0$, such that $\lim_{t\to\infty} x(t)=z$.
        \item For any $x(0) \in \rint F_{-4}$, there exists a unique $z\in \rint E_{12}$ with $a_{31} z_1+a_{32} z_2<0$, such that $\lim_{t\to\infty} x(t)=z$.
    \end{enumerate}
\end{enumerate}
\end{proposition}
\begin{proof}
    We show that any
    $x\in \{e_3,e_4\}\cup (E_{12}-L)$ is unstable. For each $i = 1,2,3,4$, define 
    $$\sigma(i):=
    \begin{cases}
     4 & \mbox{if } i = 1; \\
     3 & \mbox{if } i = 2; \\
     1 & \mbox{if } i = 3; \\
    2 & \mbox{if } i = 4. 
    \end{cases}
    $$ 
    It follows from the sign pattern of $A$ (represented by Graph~(\ref{subfig:III}) in Fig.~\ref{fig:combined_cases}) and~\eqref{eq:limit_edge} that if $x(0)\in \rint E_{i\sigma(i)}$, then $\lim_{t\to\infty} x(t) = e_{\sigma(i)} \neq e_i$. 

    We now consider the stability of any $x\in \rint E_{12} -L$. By Lemma~\ref{lem:4cycle} and~\eqref{eq:edge_kernel_condition}, $\rint E_{12} -L$ can be decomposed into two subsets as 
    % \begin{equation}
    \begin{align*}
        I_1& := \{ x \in\rint E_{12} \mid 
a_{31} x_1+a_{32} x_2>0\}, \\
I_2 & :=\{ x \in\rint E_{12} \mid 
a_{41} x_1+a_{42} x_2>0\}.
    \end{align*}
    % \end{equation}
    We first show that any $x\in I_1$ is unstable. 
    By item~(4) of Lemma~\ref{lem:face_lem}, we have that for any $x(0)\in \rint F_{-4}$ (which can be chosen arbitrarily close to the given~$x$), there exists a unique $z\in \rint E_{12}$ satisfying $a_{31} z_1+a_{32} z_2<0$ such that $\lim_{t\to\infty} x(t)=z$. Since $x\neq z$, it follows that $x$ is unstable. 
    The arguments for proving instability of $x\in I_2$ are similar to those above. 
    Using again item~(4) of Lemma~\ref{lem:face_lem}, we have that for any $x(0)\in \rint F_{-3}$, there exists a unique $z\in \rint E_{12}$ satisfying $a_{41} z_1+a_{42} z_2<0$, such that $\lim_{t\to\infty} x(t)=z$. 
    This completes our proof.
\end{proof}

\begin{proposition}
Let $\mathcal{G}_A$ be the digraph in~(\ref{subfig:IV}) of Fig.~\ref{fig:combined_cases}. The following items hold:
\begin{enumerate}
    \item In this case, $L$ intersects $\partial \Delta^3$ inside $F_{-1}$ and at $e_1$. 
    The set of equilibria of dynamics~\eqref{eq:zsreplicator} in $\partial \Delta^3$ is given by
    $$(L\cap \partial \Delta^3)\cup E_{12}\cup E_{13}\cup E_{14}.$$ 
    \item Each $x\in E_{12}\cup E_{13}\cup E_{14}-\{e_1\}$ is unstable. Each $x\in L\cap \partial\Delta^3$ is stable but not asymptotically stable. 
    \item The following hold on the faces of $\Delta^3$:
    \begin{enumerate}
        \item Any $x \in \rint F_{-1}-L$ belongs to a stable periodic orbit.
        \item For any $x(0) \in \rint F_{-2}$, there exists a unique $z\in \rint E_{14}$, with $z_1=x_1(0)$, such that $
        \lim_{t\to\infty} x(t)=z$.
        \item For any $x(0) \in \rint F_{-3}$, there exists a unique $z\in \rint E_{12}$, with $z_1=x_1(0)$, such that $
        \lim_{t\to\infty} x(t)=z$.
        \item For any $x(0) \in \rint F_{-4}$, there exists a unique $z\in \rint E_{13}$, with $z_1=x_1(0)$, such that $
        \lim_{t\to\infty} x(t)=z$.
    \end{enumerate}
\end{enumerate}
\end{proposition}

\begin{proof}
We show that any
    $x\in E_{12}\cup E_{13}\cup E_{14}-\{e_1\}$ is unstable. For each $i\in \{2,3,4\}$, define $$\sigma(i):=
    \begin{cases}
     3 & \mbox{if } i = 2; \\
     4 & \mbox{if } i = 3; \\
     2 & \mbox{if } i = 4. 
    \end{cases}
    $$ 
    It then follows from the sign pattern of $A$ (represented by Graph~(\ref{subfig:IV}) in Fig.~\ref{fig:combined_cases}) and item~(6) of Lemma~\ref{lem:face_lem} that if $x(0)\in\rint F_{-i}$, then there exists a unique $z\in E_{1\sigma(i)}$ such that $\lim_{t\rightarrow\infty}x(t)=z\notin E_{1i}$. 
\end{proof}

\begin{proposition}
    Let $\mathcal{G}_A$ be the digraph in~(\ref{subfig:V}) of Fig.~\ref{fig:combined_cases}. The following items hold:
    \begin{enumerate}
        \item 
        In this case, $L$ intersects $\partial \Delta^3$ inside $E_{12}$ and $E_{34}$.
        The set of equilibria of dynamics~\eqref{eq:zsreplicator} in $\partial \Delta^3$ is given by
        $$ E_{12}\cup E_{34}.$$ 
        \item Each $x\in E_{12}\cup E_{34}-L$ is unstable. Each $x\in L\cap \partial\Delta^3$ is stable but not asymptotically stable.
        \item The following hold on the faces of $\Delta^3$:
        \begin{enumerate}
        \item For any $x(0) \in \rint F_{-1}$, there exists a unique $z\in \rint E_{34}$ with  $a_{23}z_3+a_{24}z_4<0$, such that $\lim_{t\to\infty} x(t)=z$.
        \item For any $x(0) \in \rint F_{-2}$, there exists a unique $z\in \rint E_{34}$ with $a_{13}z_3+a_{14}z_4<0$, such that $\lim_{t\to\infty} x(t)=z$.
        \item For any $x(0) \in \rint F_{-3}$, there exists a unique $z\in \rint E_{12}$ with $a_{41}z_1+a_{42}z_2<0$, such that $\lim_{t\to\infty} x(t)=z$.
        \item For any $x(0) \in \rint F_{-4}$, there exists a unique $z\in \rint E_{12}$ with $a_{31}z_1+a_{32}z_2<0$, such that $\lim_{t\to\infty} x(t)=z$. 
        \end{enumerate}
    \end{enumerate}
\end{proposition}
\begin{proof}
We need to show that any
$x\in E_{12}\cup E_{34}-L$ is unstable. 
Decompose $\rint E_{12} -L$ into the disjoint union of $I_1$ and $I_2$ as introduced in the proof of Proposition~\ref{prop:iii}. Then, by the same arguments in that proposition, we have that any $x\in \rint E_{12} -L$ is unstable.

% We first consider the stability of any $x\in \rint E_{12} -L$.

% By Lemma~\ref{} $\rint E_{12} -L$ can be decomposed into two sets as 
% \begin{align*}
%     I_1&:=\{x\in\rint E_{12}\mid
% a_{31}x_1+a_{32}x_2>0\},\\
% I_2&:=
% \{x\in\rint E_{12}\mid
% a_{41}x_1+a_{42}x_2>0\}.
% \end{align*}
%      by Lemma~\ref{lem:twoedge4cycle}. It then follows from item~(4) of Lemma~\ref{lem:face_lem} that 
%     for any $x(0)\in \rint F_{-3}$, there exists a unique $z\in \rint E_{12}$ satisfying $a_{41} z_1+a_{42} z_2<0$, such that $\lim_{t\to\infty} x(t)=z$. Therefore, any $x\in \rint E_{12}$ with $a_{41} z_1+a_{42} z_2>0$ is unstable. Moreover, for any $x(0)\in \rint F_{-4}$, there exists a unique $z\in \rint E_{12}$ satisfying $a_{31} z_1+a_{32} z_2<0$, such that $\lim_{t\to\infty} x(t)=z$. Therefore, any $x\in \rint E_{12}$ with $a_{31} z_1+a_{32} z_2>0$ is unstable.

Similarly, we can decompose $\rint E_{34} -L$ into two disjoint subsets as follows: 
\begin{align*}
I_3:=\{z\in\rint E_{34}\mid
a_{13}z_3+a_{14}z_4>0\},
\\
I_4:=\{z\in\rint E_{34}\mid
a_{23}z_3+a_{24}z_4>0\}. 
\end{align*} 
It follows from item~(4) of Lemma~\ref{lem:face_lem} that for any $x(0)\in \rint F_{-1}$, there exists a unique $z\in \rint E_{34}$ satisfying $a_{23}z_3+a_{24}z_4<0$, such that $\lim_{t\to\infty} x(t)=z$, which implies that any $x\in I_4$ is unstable. Also,  for any $x(0)\in \rint F_{-2}$, there exists a unique $z\in \rint E_{34}$ satisfying $a_{13}z_3+a_{14}z_4<0$, such that $\lim_{t\to\infty} x(t)=z$ and hence, any $x\in I_3$ is unstable. 
\end{proof}

}
    
    The phase portraits for the five cases (I)--(V) are shown in Fig.~\ref{fig:phase_portraits}.
    
    %There exists a $z\in E_{jk}$, with $j,k\in \{1,2,3,4\}\backslash \{i\}$, and $a_{ij}\neq 0,a_{ik}\neq 0$, such that
    % $\lim_{t\to\infty} x(t)=z$; see, e.g.,~\cite[case~33 in Fig.6]{Bomze1983}. {  add the description}
% \end{proposition}

% +++++++++

% Under condition \eqref{eq:assume_pf0}, $K$ is a (nonempty) closed line segment in $\Delta^3$. Specifically,
% $$
% K \cap \rint\Delta^3 = \nll(A) \cap \rint\Delta^3 = (z^-, z^+)
% $$
% where the endpoints $z^-, z^+ \in \partial\Delta^3$. Thus, $K = [z^-, z^+]$ and $K \cap \partial\Delta^3 = \{z^-, z^+\}$.

%------------------------------------------------------------
% TAC-style: degenerate cases (edge-point / vertex-point)
% No pivot assumption
%------------------------------------------------------------

% \subsection{Conditions}\label{subsec:degenerate_edge_vertex}

% \begin{Lemma}
%     If $\nll(A)\cap \rint(\Delta^3)\neq \varnothing$, then there exist two points $z^+,z^-$ such that $\nll(A)\cap \rint(\Delta^3)=\{z^-,z^+\}$.
% \end{Lemma}

% \vspace{-0.3em}
\section{Conclusion and Outlook}
\label{sec:conclusion}

This paper studies four-strategy conservative replicator dynamics induced by constant payoff matrices. We derived necessary and sufficient conditions for coexistence by relating the payoff matrix to an associated directed graph, which yields exactly five digraph isomorphism classes. This classification also determines the possible boundary phase portraits. Moreover, we rigorously prove that, under coexistence, every non-equilibrium trajectory in $\rint\Delta^3$ is a Lyapunov-stable periodic orbit. Together, these results provide a complete characterization of the phase portraits on \(\Delta^3\) and show how the global behavior is organized jointly by the skew-symmetric structure and the topology of the underlying digraph. In our future work, we will extend the present four-dimensional analysis toward a characterization of coexistence for conservative replicator dynamics in arbitrary even dimensions.

\pagebreak

\appendix

{ 
\section{Proof of Lemma~\ref{lem:face_lem}}
\label{app:boundary_proofs}
\begin{proof}
The periodic phase portrait in item~(1) appears in~\cite[Fig.~7]{Zeeman1980} and~\cite[Section 6]{BoonePiliouras2019WINE}, where a logarithmic first integral and the Poincar\'e--Bendixson theorem are used to prove that every non-equilibrium trajectory in the relative interior is a stable periodic orbit. For item~(2) and~(3), the Lyapunov function $V(x)=-x_j$ yields the convergence to $e_j$ by applying LaSalle's invariance principle, as in the standard Lyapunov analysis of replicator dynamics~\cite{HofbauerSigmund1998,Sandholm2010}.
 
It remains to identify the limiting point in items~(4)--(6). 

For item~(4), we have $a_{jk}>0$, $a_{j\ell}<0$, and $a_{k\ell}=0$. The dynamics restricted to $F_{-i}$ is 
\begin{equation*}
\begin{aligned}
\dot{x}_j&=x_j(a_{jk}x_k+a_{j\ell}x_\ell),\\ \dot{x}_k&=-a_{jk}x_jx_k<0, \\ \dot{x}_\ell&=-a_{j\ell}x_jx_\ell>0. 
\end{aligned}
\end{equation*}
Thus, $x_k(t)$ and $x_\ell(t)$ converge by monotonicity and boundedness. Since $x_j(t)=1-x_k(t)-x_\ell(t)$, $x_j(t)$ also converges. Moreover, $x_\ell(t)\geq x_\ell(0)>0$. If $\lim_{t\to\infty}x_j(t)>0$, then $\dot{x}_\ell(t)$ would be bounded away from zero for sufficiently large $t$, contradicting the convergence of $x_\ell(t)$. Hence, there exists a $z\in E_{k\ell}$ such that $ \lim_{t\rightarrow\infty}x(t)= z.$ 

We then consider the invariant function $$ J(x):=a_{j\ell}\log x_k-a_{jk}\log x_\ell $$ with $$ \dot{J} = a_{j\ell}\frac{\dot{x}_k}{x_k} -a_{jk}\frac{\dot{x}_\ell}{x_\ell} =0. $$ Since $x_\ell(t)\geq x_\ell(0)>0$, the invariance of $J$ rules out $x_k(t)\to0$. Thus, $z_k>0$ and $z_\ell>0$, which implies that $ z\in\rint E_{k\ell}. $ Moreover, we define $$ q(t):=a_{jk}x_k(t)+a_{j\ell}x_\ell(t), $$ such that $\dot{x}_j(t)=x_j(t)q(t)$. A direct computation gives $$ \dot{q}(t) = -x_j(t)\bigl(a_{jk}^2x_k(t)+a_{j\ell}^2x_\ell(t)\bigr)<0. $$ Hence, $q(t)$ is strictly decreasing. We show that its limit is negative. Suppose, to the contrary, that
$$
q_\infty:=\lim_{t\to\infty}q(t)\geq0.
$$
Since $q(t)$ is strictly decreasing, $q(t)>q_\infty\geq0$ for every $t\geq0$. Hence,
$$
\dot{x}_j(t)=x_j(t)q(t)>0,
$$
so $x_j(t)\geq x_j(0)>0$ for all $t\geq0$, contradicting $\lim_{t\rightarrow\infty}x_j(t)=0$. Therefore,
$$
a_{jk}z_k+a_{j\ell}z_\ell=q_\infty<0.
$$ Finally, parameterize $z\in \rint E_{k\ell}$ by $z_k=s$ and $z_\ell=1-s$, where $s\in(0,1)$. The restriction of $J$ to this edge satisfies $$ \frac{d}{ds}J(z(s)) = \frac{a_{jk}s+a_{j\ell}(1-s)}{s(1-s)}. $$ It is therefore strictly decreasing on the portion of $\rint E_{k\ell}$ satisfying $$ a_{jk}z_k+a_{j\ell}z_\ell<0. $$ Since the limiting point $z$ belongs to this portion and satisfies $J(z)=J(x(0))$, it is the unique point satisfying the constraints. This proves item~(4).

For item~(5), we have $a_{jk}<0$, $a_{j\ell}<0$, and $a_{k\ell}=0$. Recall the function
$$
q(t):=a_{jk}x_k(t)+a_{j\ell}x_\ell(t),
$$
for which $\dot{x}_j(t)=x_j(t)q(t)$ and, as computed in item~(4),
$$
\dot{q}(t)
=
-x_j(t)\bigl(a_{jk}^2x_k(t)+a_{j\ell}^2x_\ell(t)\bigr)<0.
$$
Since $q(0)<0$, we have $q(t)\leq q(0)<0$, and hence
$$
x_j(t)
=
x_j(0)\exp\left(\int_0^t q(s)\,ds\right)
\leq x_j(0)e^{q(0)t},
$$
which implies that $\lim_{t\to\infty}x_j(t)=0$. Moreover, since
$$
\dot{x}_k=-a_{jk}x_jx_k>0,
\qquad
\dot{x}_\ell=-a_{j\ell}x_jx_\ell>0,
$$
we have that $x_k(t)$ and $x_\ell(t)$ converge by monotonicity and boundedness. Therefore, there exists a $z\in\rint E_{k\ell}$ such that
$\lim_{t\to\infty}x(t)=z.$

Consider the same invariant function $J$ as in item~(4) and parameterize $\rint E_{k\ell}$ by $z_k=s$ and $z_\ell=1-s$, we have
$$
\frac{d}{ds}J(z(s))
=
\frac{a_{jk}s+a_{j\ell}(1-s)}{s(1-s)}<0.
$$
Hence, $J$ is strictly decreasing on $\rint E_{k\ell}$, so
$J(z)=J(x(0))$ uniquely determines the limiting point $z$.

For item~(6), we have $a_{j\ell}=a_{k\ell}=0$ and $a_{jk}<0$. Hence, $$ \dot{x}_\ell=0,\qquad \dot{x}_j=a_{jk}x_jx_k<0,\qquad \dot{x}_k=-a_{jk}x_jx_k>0. $$ Thus, $x_\ell(t)=x_\ell(0)$, while $x_j(t)$ and $x_k(t)$ converge by monotonicity. If $\lim_{t\to\infty}x_j(t)>0$, then, since $x_k(t)\geq x_k(0)>0$, $\dot{x}_j(t)$ would be bounded above by a negative constant for all sufficiently large $t$, contradicting the convergence of $x_j(t)$. Therefore, we have $ \lim_{t\to\infty}x_j(t)=0,$ which implies that $ \lim_{t\to\infty}x(t)=z\in\rint E_{k\ell}$ with $z$ satisfying  $z_\ell=x_\ell(0), z_k=1-x_\ell(0). $  This completes our proof.
\end{proof}
}

%\nocite{*}
\printbibliography
\end{document}